\documentclass{amsproc}
\usepackage{amsmath}
\usepackage{amssymb}
\usepackage{graphicx}
\usepackage{hyperref}
\DeclareGraphicsExtensions{.png,.pdf,.eps}
\usepackage{amsthm}
\usepackage[all,2cell]{xy}
\usepackage{tikz}
\usepackage{multirow}
\usepackage{enumerate}

\newtheorem{theorem}{Theorem}[section]
\newtheorem{lemma}[theorem]{Lemma}
\newtheorem{corollary}[theorem]{Corollary}
\newtheorem{proposition}[theorem]{Proposition}

\theoremstyle{definition}

\newtheorem{notation}[theorem]{Notation}
\newtheorem*{notation*}{Notation and conventions}

\newtheorem{remark}[theorem]{Remark}

\newtheorem{set}[theorem]{Setup}

\def\C{{\mathbb C}}

\def\P{{\mathbb P}}
\def\Q{{\mathbb Q}}
\def\R{{\mathbb R}}
\def\Z{{\mathbb Z}}

\def\cC{{\mathcal C}}
\def\cD{{\mathcal D}}
\def\cE{{\mathcal E}}

\def\cL{{\mathcal L}}

\def\cN{{\mathcal{N}}}
\def\cO{{\mathcal{O}}}

\def\cV{{\mathcal V}}

\def\cY{{\mathcal Y}}
\def\Q{{\mathbb{Q}}}

\def\fg{{\mathfrak g}}
\def\fh{{\mathfrak h}}
\def\fp{{\mathfrak p}}
\def\fb{{\mathfrak b}}
\def\fa{{\mathfrak a}}
\def\fd{{\mathfrak d}}
\def\fe{{\mathfrak e}}
\def\ff{{\mathfrak f}}

\def\fsl{{\mathfrak sl}}

\def\operatorname#1{\mathop{\rm #1}\nolimits}

\def\DA{{\rm A}}
\def\DB{{\rm B}}
\def\DC{{\rm C}}
\def\DD{{\rm D}}
\def\DE{{\rm E}}
\def\DF{{\rm F}}
\def\DG{{\rm G}}

\def\Aut{\operatorname{Aut}}
\def\CAut{\operatorname{CAut}}

\def\Hom{\operatorname{Hom}}

\def\Pic{\operatorname{Pic}}
\def\Hom{\operatorname{Hom}}

\def\codim{\operatorname{codim}}

\def\rank{\operatorname{rank}}

\def\det{\operatorname{det}}
\def\ss{\operatorname{ss}}

\def\Ad{\operatorname{Ad}}
\def\ad{\operatorname{ad}}
\def\Lie{\operatorname{Lie}}

\def\PGL{\operatorname{PGL}}
\def\SL{\operatorname{SL}}
\def\Hex{\operatorname{Hex}}

\def\SO{\operatorname{SO}}
\def\Sp{\operatorname{Sp}}

\newcommand{\pb}{\ar@{}[dr]|{\text{\pigpenfont J}}}
\def\Mo{\operatorname{\hspace{0cm}M}}

\makeindex

\usepackage[textsize=tiny]{todonotes}

\newcommand\ignore[1]{}

\DeclareMathOperator{\HH}{H}
\DeclareMathOperator{\mult}{mult}

\newcommand\CC{{\mathbb{C}}}
\newcommand\PP{{\mathbb{P}}}
\newcommand\RR{{\mathbb{R}}}

\newcommand\ZZ{{\mathbb{Z}}}

\newcommand\g{{\mathfrak g}}

\newcommand\ra{{\ \rightarrow\ }}
\newcommand\lra{\longrightarrow}

\title
{High rank torus actions on contact manifolds}

\author[Occhetta]{Gianluca Occhetta}
\address{Dipartimento di Matematica, Universit\`a degli Studi di Trento, via
Sommarive 14 I-38123 Povo di Trento (TN), Italy}

\email{gianluca.occhetta@unitn.it, eduardo.solaconde@unitn.it}

\author[Romano]{Eleonora A. Romano}
\address{Instytut Matematyki UW, Banacha 2, 02-097 Warszawa, Poland}
\email{J.Wisniewski@uw.edu.pl, elrom@mimuw.edu.pl}

\author[Sol\'a Conde]{Luis E. Sol\'a Conde}

\author[Wi\'sniewski]{Jaros\l{}aw A. Wi\'sniewski}


\thanks{First and third author supported by PRIN project ``Geometria delle variet\`a algebriche''. First, second and third author supported by MIUR project FFABR. Second and fourth author supported by Polish National Science Center project 2016/23/G/ST1/04282. The results in this paper were partially obtained while the second and fourth author visited the University of Trento. }

\begin{document}
\maketitle

\begin{abstract}
We prove LeBrun--Salamon conjecture in the following situation: if $X$ is a contact Fano manifold of dimension $2n+1$ whose group of automorphisms is reductive of rank $\geq \max(2,(n-3)/2)$ then $X$ is the adjoint variety of a simple group. The rank assumption is fulfilled not only by the three series of classical linear groups but also by almost all the exceptional ones.
\end{abstract}


\section{Introduction}\label{sec:intro}

A fundamental result of Riemannian geometry is the classification of manifolds with a metric according to their holonomy groups, by De Rham and Berger, \cite{Berger}. One of the classes of Riemannian manifolds in this classification are quaternion-K{\"a}hler manifolds which are of (real) dimension $4n$, with $n\geq 2$, and holonomy group $\Sp(n)\Sp(1)=[\Sp(n)\times \Sp(1)]/\ZZ_2\subset \SO(4n,\RR)$. Salamon proved that the twistor space of a compact positive quaternion-K{\"a}hler manifold is a (complex) contact Fano manifold of (complex) dimension $2n+1$ and admits a K{\"a}hler-Einstein metric, \cite{Salamon}; a converse statement holds by a result of LeBrun, \cite{LeBrun}. A celebrated LeBrun--Salamon conjecture, \cite{SaLe}, asserts that every compact positive quaternion-K{\"a}hler manifold is symmetric or, equivalently, every contact Fano manifold with K{\"a}hler-Einstein metric is a homogeneous space. In fact, quaternion-K{\"a}hler symmetric spaces are known as Wolf spaces and, equivalently, homogeneous complex contact manifolds are known to be the closed orbits in the projectivizations of adjoint representations of simple algebraic groups; we call them adjoint varieties.

In the absence of a proof of LeBrun--Salamon conjecture in its full generality, low-dimensional cases have been verified for $n\leq 4$. That is, the conjecture was proved for positive quaternion-K{\"a}hler manifolds of (real) dimension $\leq 16$ and for contact Fano manifolds with a K{\"a}hler-Einstein metric of dimension $\leq 9$, see \cite{BWW} and references therein. Also, the conjecture is known for contact Fano manifolds $X$ (even without assuming that they admit a K{\"a}hler-Einstein metric) if the first Chern class of the quotient $L=TX/F$ of the contact distribution $F\hookrightarrow TX$ does not generate the second cohomology $\HH^2(X,\mathbb{Z})$, \cite{KPSW}. In fact, if the class is divisible in $\HH^2(X,\mathbb{Z})$ then $X$ is known to be the projective space $\mathbb{P}^{2n+1}$, hence the closed orbit in the adjoint representation of $\Sp(2n+2,\mathbb{C})$. On the other hand, if the rank of the second cohomology is $\geq 2$, then $X$ is the (naive) projectivization of the cotangent bundle on $\mathbb{P}^{n+1}$, hence the closed orbit in the adjoint representation of $\SL(n+2,\mathbb{C})$.

LeBrun--Salamon conjecture was proved in some cases with additional assumptions on group actions. Bielawski \cite{Bielawski} and  Fang, \cite{Fa1, Fa2},  considered quaternion-K{\"a}hler manifolds with an action of a (real) torus.  It is known, \cite{Matsushima}, that the automorphisms group of a Fano manifold with a K{\"a}hler-Einstein metric is reductive, which is an equivalent of a real compact group in the complex case. Thus the automorphisms of such a variety contain an algebraic torus; the rank of a maximal subtorus is, by definition, the rank of the reductive group.  The results of Fang imply that contact Fano manifolds with K{\"a}hler-Einstein metric admitting an action of an algebraic torus of rank $r \geq n/2+3$ are adjoint varieties of the simple groups $\Sp(2n+2,\mathbb{C})$ and $\SL(n+2,\mathbb{C})$, which are of rank $n+1$. We note that the condition $r\geq n/2+3$ is {\em not} fulfilled for the group $\SO(n+4,\mathbb{C})$ which is of rank $\lfloor n/2\rfloor+2$.

The main theorem of the present paper,  Theorem \ref{thm:LBFang}, improves previous results so that not only it covers the adjoint varieties of all the classical series of linear groups but also almost all the exceptional ones. That is, if $X$ is a contact Fano manifold of dimension $2n+1$ whose group of automorphisms is of rank $\geq \max(2,(n-3)/2)$ then $X$ is the adjoint variety of one of the classical linear groups $\SL(n+2,\mathbb{C})$, i.e.,~type $\DA$, $\Sp(2n+2,\mathbb{C})$, i.e.,~type $\DC$, $\SO(n+4,\mathbb{C})$, i.e., type $\DB$ or $\DD$, or a simple group of type $\DG_2, \DF_4, \DE_6, \DE_7$.\par\medskip

\begin{figure}[h!!]
\includegraphics[width=10cm]{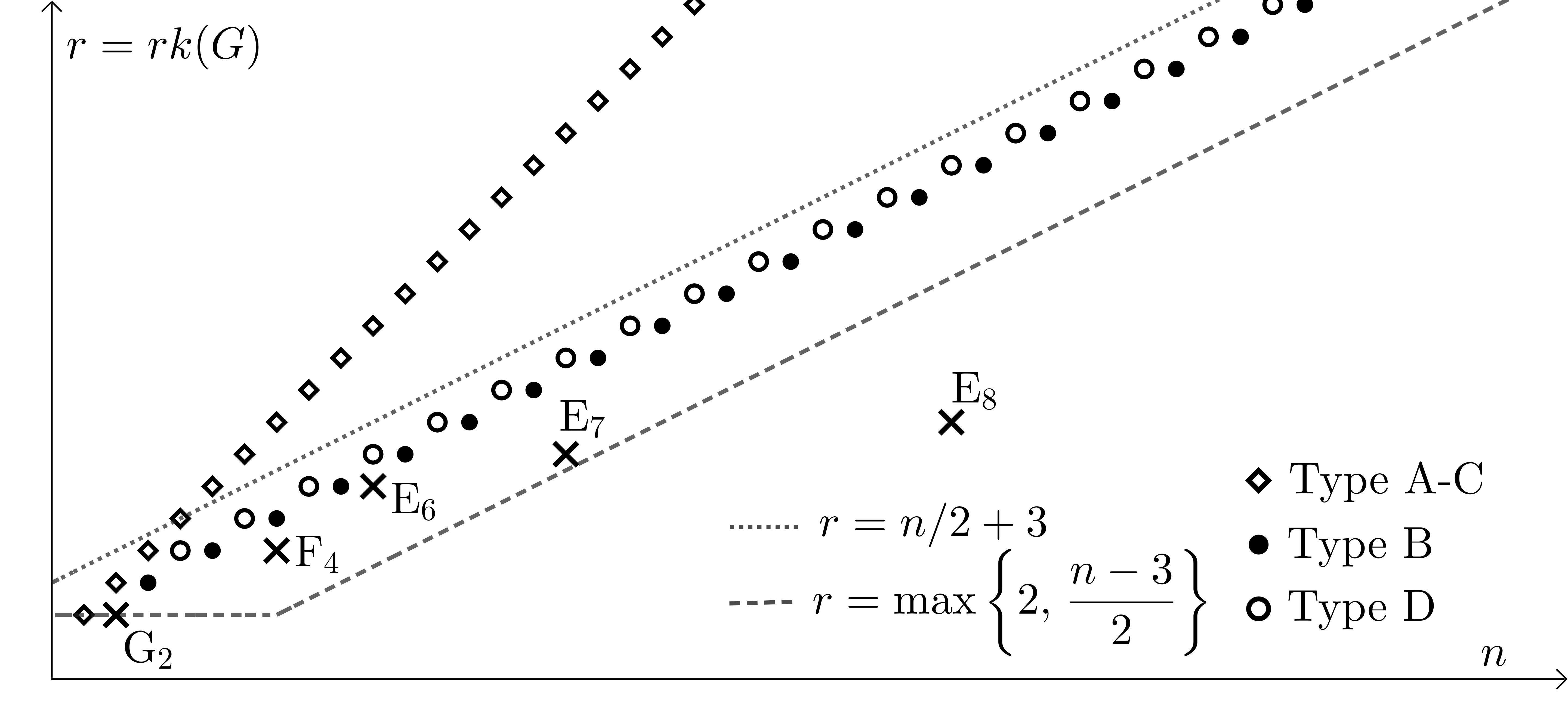}
\caption{\label{fig:fang} Comparing rank and dimension of adjoint varieties for simple groups. The line below is the bound of Theorem \ref{thm:LBFang};
the line above indicates previously known bounds.
}
\end{figure}

Although motivated by a problem from Riemannian geometry, the present paper depends solely on methods from algebraic geometry. Thus our language and formulation of results is provided in terms specific for algebraic group actions, especially algebraic torus actions, and complex birational geometry. In fact, Theorem \ref{thm:LBFang} follows from Theorem \ref{thm:LB_isolated_pts}, which is the technical core of the present paper. Theorem \ref{thm:LB_isolated_pts} asserts that a contact Fano manifold is the adjoint variety of one of the simple groups if the group of its automorphisms is reductive of rank $\geq 2$ and, in addition, the action of its Cartan torus has only isolated {\em extremal fixed points}. The arguments in the proof of this result are about the action of a suitably chosen two dimensional subtorus in the Cartan torus in question. This {\em downgrading} of the torus action and suitable results about {\em adjunction theory} for varieties with a $\CC^*$-action, \cite{RW}, and about the birational  geometry of {\em small bandwidth} $\CC^*$-actions, \cite{ORSW}, enable to identify the variety as the adjoint variety of a simple group.


The paper is organized as follows. Section 2 contains general results about algebraic groups and their actions, while Section 3 concentrates on homogeneous manifolds. In particular, we deal with adjoint varieties of simple algebraic groups. As the main result of the present paper covers the case of adjoint varieties of most exceptional algebraic groups, a significant part of its contents concerns properties specific for these groups, see e.g.~\cite{LM01,KaYa}. However, in our approach, the classical linear group series $\SO(n+4,\CC)$, that is the groups of type $\DB$ and $\DD$, are treated on equal terms with exceptional groups. This is because the Picard groups of their adjoint varieties are generated by the class of the quotient $L=TX/F$, contrary to the other two types of classical linear groups. This allows us to restate geometrical properties of adjoint varieties related to the so-called Freudenthal magic square in a context appropriate for our purposes.

Section 4 summarizes specific results concerning torus actions on contact manifolds, as in \cite[Section~5]{BWW},  while Section 5 contains the proof of Theorem \ref{thm:LB_isolated_pts}, the main technical result of the paper. For clarity, the proof is divided into five steps; a short guide through the proof is provided after the formulation of the theorem. Finally, Section 6 contains the proof of Theorem \ref{thm:LBFang}.
\medskip

\noindent{\bf Acknowledgements:} The authors would like to thank A. Weber and J. Buczy\'nski for their interesting comments and for providing bibliographic references. They are also grateful to Willem de Graaf for his comments on Lie algebras.



\section{Preliminaries}\label{sec:prelim}

This section contains some background material and language we will use later on. Let us start by introducing some general notation.

\begin{notation*} Throughout this paper all the varieties will be algebraic, projective and defined over the field of complex numbers, mostly smooth (in this case we will also refer to them as projective manifolds). Given a vector bundle $\cE$ on a variety $X$, we will denote by $\P(\cE)$ the quotient of the complement of the zero section of $\cE$ by the action of homotheties.  The Lie groups $\SL(n,\C),\SO(n,\C),\dots$ will be denoted by $\SL(n),\SO(n),\dots$ and considered with their structure of algebraic groups. Given a semisimple group $G$, a {\em rational homogeneous $G$-variety} is a projective variety of the form $G/P$; a subgroup $P\subset G$ for which $G/P$ is projective is called {\em parabolic}. Such a variety is completely determined by combinatorial data: the Dynkin diagram $\cD$ of $G$ and a subset $I$ of the set $D$ of nodes in $\cD$ (of cardinality equal to the Picard number of $G/P$). More concretely, up to conjugacy, the parabolic subgroup $P$ can always be written as $BW(D\setminus I)B$, where $B\subset G$ is a Borel subgroup, and $W(D\setminus I)$ is the subgroup of the Weyl group $W$ of $G$ generated by the reflections associated to the nodes of $D\setminus I$.  
We will then write (see \cite[Section~2.4]{ORSW} for details and examples):
$$
\cD(I):=G/BW(D\setminus I)B.
$$  
\end{notation*}

\subsection{Torus actions on smooth projective varieties}\label{ssec:prelimXH}

We will briefly recall here some general facts on torus actions on smooth complex projective varieties. We refer the interested reader to \cite{BWW} for a complete account on the concepts and tools that we introduce here. 

Throughout the section we will denote by $(X,L)$  a {\em polarized pair}, consisting of a smooth complex projective variety $X$ and an ample line bundle $L$ on $X$, and by $H\simeq {(\C^*)}^{r}$ an $r$-dimensional complex torus, acting effectively on $X$. The {\em lattice of weights} of $H$ will be denoted by $\Mo(H):=\Hom(H,\C^*)$.

The set of points of $X$  fixed by the action of $H$, denoted by $X^H$, is a union of smooth closed irreducible subvarieties, indexed by a finite set $\cY$:
$$
X^H=\bigsqcup_{Y\in \cY} Y.
$$

Given a {\em linearization of the action of} $H$ on the line bundle $L$, $\mu_L\colon H\times L\rightarrow L$, the weight of the action of $H$ on the fiber $L_x$ on a fixed point $x$ depends only on the fixed point component $Y$ containing $x$, and we denote it, abusing notation, by $\mu_L(Y)$.   
  
The {\em polytope of fixed points} $\bigtriangleup(X,L,H,\mu_L)$ (respectively, the {\em polytope of sections} $\Gamma(X,L,H,\mu_L)$) is defined as the convex hull in $\Mo(H)\otimes \mathbb{R}$ of the weights of the action of $H$ on $L$ (respectively, of the weights of the action of $H$ on $H^{0}(X,L)$).  We refer to \cite[Section 2.1]{BWW} for details. A fixed point component $Y\in\cY$ for which $\mu_L(Y)$ is a vertex of $\bigtriangleup(X,L,H,\mu_L)$ will be called an {\em extremal fixed point component}.  
  
We will sometimes consider the restriction of the action of $H$ on $(X,L)$ to a subtorus $H'\subset H$, a process that we call {\em downgrading}. We denote by $\imath:H'\to H$ the inclusion, by $\imath^*:\Mo(H)\to\Mo(H')$ the induced map between the lattices of weights, by  $\cY'$ the set of fixed point components of the induced action of $H'$ on $X$, and by $\mu'_L\colon \cY'\to \Mo(H')$ the map associating to every $H'$-fixed point component $Y$ the $H'$-weight of $L$ on every point of $Y$. The following straightforward lemma describes how the weights of the action behave with respect to the downgrading:

\begin{lemma}\label{lem:downgrading}
With the above notation:
\begin{enumerate}[(i)]
\item every $Y\in \cY$ is contained in a unique fixed point component $Y'\in \cY'$, and $\mu'_L(Y')=\imath^*\mu_L(Y)$;
\item every $Y'\in \cY'$ is invariant by $H$, and contains at least a fixed point component $Y\in \cY$;
\item given $Y'\in \cY'$, the torus $H/H'$ acts on $Y'$ with fixed point locus $(Y')^{H/H'}=$ $X^H\cap Y'=Y'^H$; 
\item $\bigtriangleup(X,L,H',\mu'_L)=\imath^*\!\!\bigtriangleup\!(X,L,H,\mu_L)$;
\item for every $Y'\in \cY'$, $\bigtriangleup(Y',L,H/H',{\mu_L})
\subseteq\bigtriangleup(X,L,H,\mu_L)\cap (\imath^{*})^{-1}\mu'_L(Y')$.
\end{enumerate} 
\end{lemma} 

Given a fixed point component $Y\in \cY$, the action of $H$ on the normal bundle $\cN_{Y|X}$ of $Y$ in $X$ provides a splitting of this bundle into eigen-subbundles:
$$
\cN_{Y|X}=\bigoplus_i \cN^{-\nu_i(Y)}(Y),
$$
whose corresponding weights, denoted by $-\nu_i(Y)$, are nontrivial. Denoting by $r_i$ the rank of $\cN^{-\nu_i(Y)}(Y)$, the set of the elements $\nu_i(Y)$, counted $r_i$-times, will be called the {\em compass of the action of $H$ on  $Y$}. We will write it as $$\mathcal{C}(Y,X,H)=\left\{\nu_1(Y)^{r_1},\nu_2(Y)^{r_2},\dots\right\}.$$

In order to make our exposition self-contained, we will recall from \cite{BWW} two statements regarding compasses that we will use later on; the first one describes the behavior of the compass with respect to downgrading:

\begin{lemma}  \cite[Lemma 2.13]{BWW} \label{lem:compass_prop} Let  $H$ be an algebraic torus acting on a smooth projective variety $X$, and $H'\subset H$ a subtorus. Let $\imath^*\colon \Mo(H)\to \Mo(H')$ be the projection between the corresponding lattices of weights. Take fixed point components $Y'\subset X^{H'}$ and $Y\subset Y'\cap X^{H}$. Then:
\begin{enumerate}[(i)] 
\item $\cC(Y',X,H')=\imath^*(\cC(Y,X,H))$; 
\item $\cC(Y,Y',H/H')=\cC(Y,X,H)\cap \ker{\imath^*}$, and $\cN_{Y|Y'}=\bigoplus_{\nu\in \cC(Y,Y',H/H')} \cN^{-\nu}(Y)$.
\end{enumerate}
\end{lemma}

Another important property of the compass is that its elements at a given fixed point $x\in X^H$ are vectors pointing from $\mu_L(x)$ towards the weight $\mu_L(x')$ of another fixed point $x'\in X^H$. More precisely:

\begin{lemma} \cite[Corollary 2.14]{BWW} \label{lem:compass_prop2} Let $H$ be an algebraic torus acting on a smooth projective variety $X$, let $L$ be a line bundle on $X$, and let  $\mu_L$ be a linearization of the action. Take $Y\subset X^{H}$, and $\nu \in \cC(Y,X,H)$. Then there exists $Y^{\prime}\subset X^{H}$ and $\lambda \in \mathbb{Q}_{>0}$ such that $\mu_{L}(Y^{\prime})=\mu_{L}(Y)+\lambda \nu$.
\end{lemma}

The above invariants of $H$-actions were introduced in \cite{BWW} with the motivation that, in the case in which $H\subset G$ is the maximal torus of a semisimple group, they can be used to characterize $G$-representations and rational homogeneous $G$-varieties. In particular, in this paper we will make use of the following statement, which has been proved in \cite[Propositions~2.23, 2.24]{BWW} under the additional assumption that all the fixed point components are isolated points; the proof follows the same line of argumentation.

\begin{proposition} \label{prop:BWW224} 
Let $(X,L)$ and $(X',L')$ be two polarized pairs.  Let $H$ be a complex torus acting on both pairs and denote by $\cY$ and $\cY'$ the set of irreducible fixed components of $X^H$ and $X'^H$, respectively. Assume that there exists a bijection $\psi:\cY\to \cY'$ such that, for every $Y\in \cY$:
\begin{itemize}
\item[{$(A_1)$}] $Y\simeq \psi(Y)$;
\item[{$(A_2)$}] $\mu_L(Y)=\mu_{L'}(\psi(Y))$, and $L_{|Y}\simeq L'_{|\psi(Y)}$;  
\item[{$(A_3)$}] $\cC(Y,X,H)=\cC(\psi(Y),X',H)$, and $\cN^{-\nu}(Y)\simeq\cN^{-\nu}(\psi(Y))$ for every $\nu\in\cC(Y,X,H)$.
\end{itemize}
Then:
\begin{itemize}
\item[{$(C_1)$}] If $\HH^i(X,L)=\HH^i(X',L')=0$ for $i>0$, then $\HH^0(X,L)$ is $H$-equivariantly isomorphic to $\HH^0(X',L')$.
\item[{$(C_2)$}] If the actions of $H$ on $X$ and $X'$ are restrictions of the actions of a semisimple group $G$ containing $H$ as a maximal torus, and if $X'$ is a rational homogeneous $G$-variety, then $(X,L)\simeq (X',L')$.
\end{itemize}
\end{proposition}

\begin{proof}
Note that conditions $(A_2), (A_3)$ are equivalent to require that $L_{|Y}$ is $H$-equivariantly isomorphic to $L'_{|\psi(Y)}$ and that $\cN_{Y|X}$ is $H$-equivariantly isomorphic to $\cN_{\psi(Y)|X'}$, respectively. This implies an equality of $H$-equivariant Euler characteristics $\chi^H(X,L)=\chi^H(X',L')$ (see \cite[Theorem~A.1]{BWW}); together with the hypothesis on the vanishing of the cohomology, this tells us that $\HH^0(X,L)$, $\HH^0(X',L')$ are equal as elements of the representation ring of $H$. 

For the second part, a similar argument provides isomorphisms of $H$-modules: 
$H^0(X,L^{\otimes m})\simeq H^0(X',L'^{\otimes m})$, for $m\gg 0$.
By standard Representation Theory (cf \cite[Theorem~14.18]{FH}), since these spaces are $G$-modules, they are isomorphic also as $G$-modules. Then for $m\gg 0$, the natural morphism $\phi: X\to \P(H^0(X,L^{\otimes m})^\vee)$ will be a $G$-equivariant embedding so that the image of $X$ will be invariant by $G$. In particular, it will contain the unique closed orbit of the action, which is isomorphic to $X'$. Since condition $(A_3)$ implies that $\dim{X}=\dim{X'}$ we conclude that $(X,L)\simeq (X',L')$. 
\end{proof}

\subsection{Actions of $\C^*$ on smooth projective varieties}\label{ssec:actionC*}
  
The case of actions of the torus $\C^*$, that has been extensively studied in the literature (see \cite{CARRELL} and references therein),  will be particularly useful for our purposes. We introduce here some notation and basic facts we will use when dealing with this type of actions. 
 
We first choose an isomorphism $\Mo(\C^*)\simeq \Z$. Given a $\C^*$-action on a smooth projective variety $X$, for every fixed point component $Y\in \cY$ the normal bundle of $Y$ in $X$ splits into two subbundles, on which $\C^*$ acts with positive and negative weights, respectively:
\begin{equation}
\cN_{Y|X}\simeq \cN^+(Y)\oplus \cN^-(Y).
\label{eq:normal+-}
\end{equation}
We denote by $\nu^\pm(Y)=\rank \cN^\pm(Y)$ the ranks of these subbundles. Then the following result, due to Bia{\l}ynicki-Birula  (cf. 
\cite{BB}, \cite[Theorem 4.4]{CARRELL}), allows us to compute the homology groups of $X$:

\begin{theorem}
\label{thm:BB_decomposition}
In the situation described above there is a decomposition:
 $$H_m(X,\ZZ)=\bigoplus_{Y\in\cY}H_{m-2\nu^+(Y)}(Y,\ZZ)= \bigoplus_{Y\in\cY}H_{m-2\nu^-(Y)}(Y,\ZZ), \mbox{ for every }m.$$ 
\end{theorem} 

Given a polarized pair $(X,L)$ admitting a $\C^*$-action, we identify the weights $\mu_L(Y)$ with the corresponding integers. The minimum and maximum value of $\mu_L$, denoted by $\mu_{\min}$ and $\mu_{\max}$, are achieved in two unique fixed point components, that we call the \textit{sink} and the \textit{source} of the action; the rest of the fixed point components will be called \textit{inner} the fixed points components which are neither the source nor the sink.  The \textit{bandwidth} of the action of $H$ on $(X,L)$ is defined as the difference $|\mu|=\mu_{\max}-\mu_{\min}$. Moreover, we say that the action of $H$ on $X$ is {\em  equalized at} $Y$ if $H$ acts on $\cN_{Y|X}$ with weights $\pm 1$, and we call the action {\em  equalized} if it is equalized at every fixed point component $Y\in \cY$. Note that if the action is equalized at the sink and the source and has bandwidth two and three, then using \cite[Lemma 3.1]{RW} easily follows that the action is equalized (cf. \cite[Corollary 2.14]{ORSW} for the case of bandwidth two).  \\
When the bandwidth of a $\C^*$-action on a pair is small, one expects to have reasonably short lists of examples, under certain assumptions. We refer to \cite{ORSW} for an account on this matter. In this paper we will make use of the following two results, regarding equalized actions of bandwidth two and three. The first tells us that under some conditions an action of bandwidth two is determined by its local behaviour at the sink and the source:

\begin{theorem}\label{thm:uniqueBW2}\cite[Corollary 5.12]{ORSW} Let $(X,L)$ be a polarized pair supporting a  $\C^*$-action of bandwidth two, equalized at the sink $Y_{-1}$ and the source $Y_1$, which are both positive dimensional. Assume moreover that $\rho_X=1$, that there exists an inner fixed point component, and that the vector bundles $\cN_{Y_{\pm1}/X}^\vee \otimes L$ are semiample. Then $X$ is uniquely determined by $(Y_{\pm1}, \cN_{Y_{\pm1}/X})$.
\end{theorem}

We will also make use  of the following complete classification of equalized bandwidth $3$ actions with isolated extremal fixed points: 

\begin{theorem}\label{thm:bw3}\cite[Theorem~4.5]{RW}, \cite[Theorem 6.8]{ORSW} 
  Let $(X,L)$ be a polarized pair, where $X$ is a projective manifold of
  dimension $n\geq 3$ with a linearized action of $\C^*$ of bandwidth three, such that its sink and source are isolated points.  Assume in addition that the action is equalized at the sink and the source, and denote by $Y_i$ the union of the inner fixed point components of weight $i$, $i=1,2$. Then one of the following holds:
  \begin{itemize}
 \item [(1)] $X=\P(\cV^\vee)$, 
 with $\cV=\cO_{\P^1}(1)^{n-1}\oplus\cO_{\P^1}(3)$, or $\cO_{\P^1}(1)^{n-2}\oplus\cO_{\P^1}(2)^2$, and $L= \cO_{\P(\cV^\vee)}(1)$. Moreover  $(Y_i,L_{|Y_i}) \simeq (\P^{n-2}, \cO_{\P^{n-2}}(1))$.

  \item [(2)] $X=\P^1\times\Q^{n-1}$, $L=\cO(1,1)$, each $Y_i$ is the disjoint union of a smooth quadric $\Q^{n-3}$ and a point, and $L_{|\Q^{n-3}} \simeq \cO_{\Q^{n-3}}(1)$.
  \item [(3)]  $X$ is one of the following rational homogeneous varieties:
$$\DC_3(3),\,\,\, \DA_5(3),\,\,\,  \DD_6(6),\,\,\, \DE_7(7),$$
$L$ is the ample generator of $\Pic(X)$ and the varieties $Y_i$ are, respectively
$$\DA_2(2),\,\,\, \DA_2(2) \times \DA_2(1),\,\,\,  \DA_5(2),\,\,\, \DE_6(1).$$
The restriction of $L$ to $Y_i$ is the ample generator of $\Pic(Y_i)$, except in the case $Y_i \simeq \DA_2(2) \simeq \P^2$, in which $L_{|Y_i} \simeq \cO_{\P^2}(2)$.
  \end{itemize}
\end{theorem}



\section{Torus actions on adjoint varieties}\label{sec:torusGP}

In this section we will describe torus actions on rational homogeneous varieties, paying special attention to the case of adjoint varieties, whose characterization is the goal of this paper. We will start by briefly recalling the action of the maximal tori of their defining semisimple groups, then we will focus on some particular downgradings of those actions, that we will use later in the proof of Theorem \ref{thm:LB_isolated_pts}.

\subsection{The action of a maximal torus}\label{ssec:maximal}  

Let $G$ be a semisimple algebraic group, $B\subset G$ a Borel subgroup, and $H\subset B$ a Cartan subgroup. We denote by $\Phi$ the root system of $G$ with respect to $H$, by $W=N_G(H)/H$ the Weyl group of $G$, by $D=\{\alpha_1,\dots,\alpha_r\}$ the base of positive simple roots of $\Phi$ induced by $B\supset H$, by $\Phi^+$ the set of positive roots determined by $B$, and by $\cD$ the Dynkin diagram of $G$. 

The following well known statement (see \cite[Section 3.4]{CARRELL}) describes the set of $H$-fixed point of every rational homogeneous variety $G/P$, with $P\supset B$; we include its proof for lack of references:

\begin{lemma} 
\label{lem:fixGP} 
The set of fixed points of $G/P$ by the action of $H$ is: 
$$(G/P)^H=\{wP,w\in W\}. $$ 
\end{lemma}  

\begin{proof}
Let $B$ be a Borel subgroup of $G$, with $H\subset B\subset P$, and let us compute first $(G/B)^H$. A point $gB$ is fixed by $H$ if and only if $gBg^{-1}\supset H$. Following \cite[Section~27]{Hum1},  this holds if and only if $gBg^{-1}$ can be written as $wBw^{-1}$, for some $w\in N_G(H)$, and this Borel subgroup depends only on the class of $w$ modulo $H$. We conclude that $(G/B)^H=\{wB|\,\,w \in W\}$. 

Now, given $P\supset B$, we note first that the natural projection $\pi_P:G/B\to G/P$ is $H$-equivariant, hence $\pi_P((G/B)^H)\subset(G/P)^H$;  the converse follows from the fact that $\pi_P$ is projective, therefore the inverse image of a fixed point of $G/P$, which is $H$-invariant, contains a fixed point.
\end{proof}

\begin{remark}\label{rem:fixGP}
The first part of the above proof shows that $(G/B)^H$ is bijective to the Weyl group of $G$, $W:=N_G(H)/H$. Moreover, if $P=BW(D\setminus I)B$, where $W(D\setminus I)\subset W$ is the subgroup generated by the reflections corresponding to a subset $D\setminus I\subset D$, then $(G/P)^H$ is bijective to the quotient $W/W(D\setminus I)$.
\end{remark} 

We describe now the compasses of the $H$-action at fixed points $wP$, $w\in W$. 

\begin{lemma}\label{lem:compassGP}
If $P=BW(D\setminus I)B$, denoting by $\Phi^+(D\setminus I)$ the set of positive roots of $\Phi$  that are linear combinations of elements $\alpha_j$, $j\in D\setminus I$,  then:
$$
\cC(wP,G/P,H)\simeq w(\Phi^+\setminus\Phi^+(D\setminus I)).
$$
\end{lemma}
\begin{proof}
As above, we choose for every $w\in W$ a preimage in $N_G(H)$, and denote it by $w$. Being $\fg$ and $\fp$ respectively the Lie algebra of $G$ and $P$, we may write isomorphisms of $H$-modules:
$$
T_{G/P,wP}\simeq \Ad_w(T_{G/P,P})\simeq \Ad_w(\fg/\fp)\simeq \fg/\!\Ad_w(\fp),
$$
Then we use the Cartan decomposition to split $\fg/\fp$ as a direct sum of $H$-eigenspaces 
$$\fp=\fh\oplus\bigoplus_{\alpha\in\Phi^+}\fg_\alpha \oplus\bigoplus_{\alpha\in\Phi^+(D\setminus I)}\fg_{-\alpha}, \qquad\fg/\fp=\bigoplus_{\alpha\in\Phi^+\setminus\Phi^+(D\setminus I)}\fg_{-\alpha}, 
$$
and conclude by noting that $\Ad_w(\fg_\alpha)=\fg_{w(\alpha)}$ for all $\alpha\in\Phi$.
\end{proof}

Let us finally consider the case in which $\fg$ is simple (abusing notation, we will say in this case that the semisimple group $G$ is {\em simple}), and let $X_G$ be the corresponding {\em adjoint variety of $G$}, that is the closed orbit of the action of $G$ in the projectivization of the adjoint representation of $G$, $X_G\hookrightarrow\PP(\g)$. The following result, in which $\beta$ denotes  the highest weight of the adjoint representation, and $v\in \fg_{\beta}$ is a nonzero eigenvector, is a consequence of Lemma \ref{lem:fixGP}.

\begin{corollary}\label{cor:adjointfix}
The set of fixed points of $X_{G}$ by the action of $H\subset G$ is:
$$
X_G^H=\{w[v],w\in W\}, 
$$
where $\mu_L(w[v])=w(\beta)$, for every $w\in W$. Moreover $X_G^H$ is bijective to the set of long roots of $G$. 
\end{corollary} 

\begin{proof}
By \cite[Claim~23.52]{FH} we know that $X_G$ is the $G$-orbit of the class of a highest weight vector $[v]$  of the adjoint representation; that is $v$ is a nonzero element of $\fg_\beta$, with $\beta$ the longest positive root of $G$ with respect to a base of simple roots. Then Lemma \ref{lem:fixGP} gives us the description of $X_G^H$. The second part follows from the fact that $\Ad_w(v)\in\g_{w(\beta)}$ for every $w\in W$. We conclude by noting that $W$ acts transitively on the set of long roots of $G$ (cf. \cite[Lemma~C]{Hum2}). 
\end{proof}

\subsection{Some special downgradings of ranks one and two}\label{ssec:downgradings}
In this section we keep the same notations introduced above, and assume that $G$ is simple. 
Let us denote by $\bigtriangleup(G)\subset\Mo(H)\otimes \mathbb{R}$ the {\em root polytope of $G$}, that is the polytope generated by the roots of $G$. We will now define some special downgradings of the action of $H$ on $X_G$, that we will use later on.
 
\begin{lemma}\label{lem:reduction}
Let $G$ be a simple group of type $\DB_r$, $r\geq 3$, $\DD_r$, $r\geq 4$, $\DE_r$, $r=6,7,8$,  $\DF_4$, or $\DG_2$, and let $\alpha,\alpha'$ be two long roots of $G$ forming an angle of $2\pi/3$ radiants. Then there exists a subgroup $S_2\subset G$ isomorphic to $\SL(3)$, inducing a projection of root polytopes $\imath^*:\bigtriangleup(G)\rightarrow\bigtriangleup(S_2)$ such that $\{\imath^*(\alpha),\imath^*(\alpha')\}$ is a base of positive simple roots of $S_2$, and  such that $(\imath^*)^{-1}(\beta)$ consists of one point for every vertex $\beta$ of the hexagon $\bigtriangleup(S_2)$.
\end{lemma} 

\begin{proof}
In the case $\DG_2$ it is known, by Borel--de Siebenthal theory \cite{Borel1949}, that the set of long roots of $\DG_2$ is a closed root subsystem of the root system of $\DG_2$ determining a subgroup $S_2\subset\DG_2$ isomorphic to $\SL(3)$. Note that $S_2$ and $\DG_2$ have the same maximal torus $H$  and the same root polytopes.

In all the remaining cases we  may always choose the long roots $\alpha,\alpha'$ forming an angle of $2\pi/3$ among the elements of the base $D=\{\alpha_1,\dots,\alpha_r\}$ of simple roots of $G$ (see Remark \ref{rem:reduction} below); let us denote them by $\alpha_i,\alpha_j$. 

We consider the subgroup $W'\subset W$ of the Weyl group $W$ of $G$ generated by the reflections $r_i$, $r_j$ (corresponding to $\alpha_i$ and $\alpha_j$), and the parabolic subgroup $P=BW'B\subset G$. We then consider a Levi decomposition of $P$, $P=U\rtimes L$, where $U$ is the unipotent radical of $P$, and $L$ is reductive,  and the commutator $S_2=[L,L]\subset L\subset P\subset G$, which is semisimple. The maximal torus of $S_2$ is $H_2:=H\cap S_2$. By construction, denoting by $\fh$, $\mathfrak{s}_2$, $\fh_2$ the Lie algebras of $H$, $S_2$ and $H_2$, by $\imath^*:\Mo(H)\otimes_\Z\R\to\Mo(H_2)\otimes_\Z\R$ the linear map induced by the inclusion $H_2\hookrightarrow H$, and by $\fg_\beta\subset\fg$ the eigenspace associated to a root $\beta$ of $G$, we may write:
$$
\mathfrak{s}_2=\fh_2\oplus \bigoplus_{\beta\in \Phi\cap (\Z \alpha_i+ \Z\alpha_j)}\fg_{\beta}.
$$ 
Then the root system of $S_2$ is $\Phi_{S_2}=\imath^*\{\pm\alpha_i,\pm\alpha_j,\pm(\alpha_i+\alpha_j)\}\subset \Mo(H_2)\otimes_\Z\R$, and $\mathfrak{s}_2$ is isomorphic to $\fsl_3$. In particular, the subgroup $S_2\subset G$ is isomorphic either to $\SL(3)$ or to $\PGL(3)$, and $\dim H_2=2$. 

Note that $\imath^*$ is an orthogonal projection and sends the lattice $\Mo(H)$ to $\Mo(H_2)$. Moreover, since all the roots in $\Phi\cap (\Z \alpha_i+ \Z\alpha_j)$ are long, then $\imath^*$  sends any root of $G$ which is not in $\Phi\cap (\Z \alpha_i+ \Z\alpha_j)$ to the interior of $\bigtriangleup(S_2)$. This shows that $\bigtriangleup(G)\rightarrow\bigtriangleup({S_2})$ has one point fibers over the vertices of the hexagon $\bigtriangleup({S_2})$.

Finally we note that in all the cases there exists a root of $G$ not in $\Phi\cap (\Z \alpha_i+ \Z\alpha_j)$ and not orthogonal to the subspace generated by $\alpha_i,\alpha_j$, hence its projection to $\Mo(H_2)\otimes_\Z\R$ is a nonzero lattice point in the interior of $\bigtriangleup({S_2})$. This shows that $\Mo(H_2)$ contains properly the root lattice of $S_2$, so necessarily ${S_2}\simeq\SL(3)$. 
\end{proof}

\begin{remark}\label{rem:reduction}
In the setting of Lemma \ref{lem:reduction}, for the cases different from $\DG_2$ we may use the following particular choice of the pairs $(i,j)$ defining $S_2\subset G$:

\begin{table}[h!]
\begin{tabular}{|c||c|c|c|c|c|c|}
\hline
 $\fg$&$\DB_r$&$\DD_r$&$\DE_6$&$\DE_7$&$\DE_8$&$\DF_4$\\\hline\hline
$\,(i,j)\,$&$\,(1,2)\,$&$\,(1,2)\,$&$\,(4,2)\,$&$\,(3,1)\,$&$\,(7,8)\,$&$\,(2,1)\,$\\\hline
\end{tabular}
\end{table}
Here we are following the standard reference \cite[p.~58]{Hum2} for the numbering of nodes of the corresponding Dynkin diagrams. 
\end{remark}

\begin{notation}\label{notn:downgradings}
Let us now describe the downgradings of the action of $H\subset G$ on the adjoint variety $X_G$ that we are going to consider in this paper.\par
\medskip

\begin{itemize}
\item[($S_2$)] With the notation of Lemma \ref{lem:reduction}, we will consider  a subgroup $S_2\subset G$, isomorphic to $\SL(3)$, a maximal torus $H_2\subset S_2$ (obtained by intersecting the maximal torus $H\subset G$ with $S_2$) given by the choice of two long roots $\alpha,\alpha'$ as in Remark \ref{rem:reduction}, for the cases different from $\DG_2$. In Figure \ref{fig:Hex0} we represented the points of $\Mo(H_2)\cap \bigtriangleup(S_2)$, which are the possible images of the roots of $G$ via $\imath^*$.\par
\medskip
\item[($S_1$)] Given a root $\alpha\in\Phi_{S_2}$ we may find a unique subgroup $S_1\subset S_2$ isomorphic to $\SL(2)$ whose Lie algebra contains the eigenspace $\fg_{\alpha}$. We denote the maximal torus of $S_1$ by $H_1= H_2\cap S_1$, and consider the projection $\pi^*:\Mo(H_2)\to\Mo(H_1)$ associated to the inclusion $\pi:H_1\hookrightarrow H_2$; it sends $\bigtriangleup(S_2)$ to  $\bigtriangleup(S_1)=[-2,2]$, and by choosing an appropriate isomorphism $\Mo(H_1)\simeq\Z$ we may write $\pi^*(\alpha)=2$ (see Figure \ref{fig:Hex0}).
\end{itemize}
\end{notation}

\begin{figure}[h!]
\includegraphics[height=4.5cm]{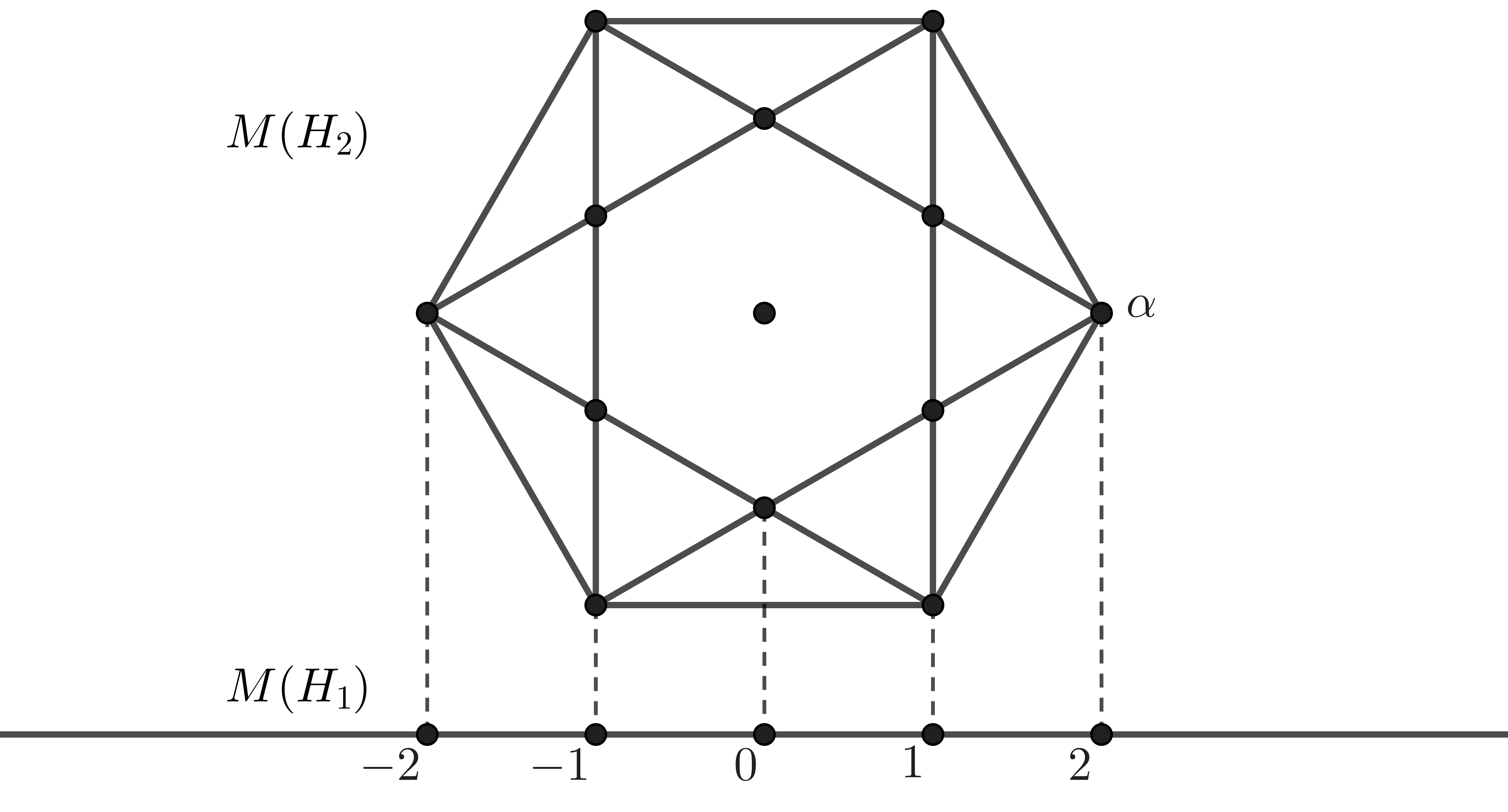} 
\caption{Lattice points in the polytopes $\bigtriangleup(S_2)$ and $\bigtriangleup(S_1)$.}
\label{fig:Hex0}
\end{figure}

Let us denote the corresponding lattices of weights by $M:=\Mo(H)$, $M_2:=\Mo(H_2)$, $M_1:=\Mo(H_1)$, and by $\imath^*:M\to M_2$, $\pi^*:M_2\to M_1$ the induced projections. The adjoint actions of $H_2$ and $H_1$ on $\g$ provide gradings of $\g$ with respect to $M_2$ and $M_1$:
$$
\g=\bigoplus_{m\in M_2}\g^2_m=\bigoplus_{m\in M_2}\left(\bigoplus_{\substack{\gamma\in\Phi\cup\{0\}\\\imath^*(\gamma)=m}}\g_\gamma\right),\qquad \g=\bigoplus_{m\in M_1}\g^1_m=\bigoplus_{m\in M_1}\left(\bigoplus_{\substack{\gamma\in\Phi\cup\{0\}\\(\imath\circ\pi)^*(\gamma)=m}}\g_\gamma\right)
$$
where we recall that $\Phi\subset M$ is the root system of $G$.
By our choice of $S_1\subset S_2\subset G$ we have that, for $i=1,2$, $\dim(\fg^i_m)=1$ whenever $m$ is a root of $S_i$: these are precisely the vertices of the root polytope $\bigtriangleup(S_i)$. The rest of values of $m\in M_i$ for which $\fg_m^i\neq 0$ correspond to inner points of $\bigtriangleup(S_i)$. 

Furthermore, the vector spaces $\fg^i_m$ are representations of the subalgebra $\fg_0^i\subset\fg$ of elements of $M_i$-degree $0$ and of its semisimple part $\g_0^{i,\ss}$, $i=1,2$. 
With the choice of the projection $\imath:M\to M_2$ presented in Remark \ref{rem:reduction}, one obtains the following description of those subalgebras:

\begin{table}[h!]
\begin{tabular}{|c||c|c|c|c|c|c|c|}
\hline
 $\fg$&$\DB_r$&$\DD_r$&$\DE_6$&$\DE_7$&$\DE_8$&$\DF_4$&$\DG_2$\\\hline\hline
$\,(i,j)\,$&$(1,2)$&$(1,2)$&$(2,4)$&$\,(1,3)\,$&$\,(7,8)\,$&$\,(1,2)\,$&\\\hline
$\fg_0^{1,\ss}$&$\,\DA_1\times\DB_{r-2}\,$&$\,\DA_1\times\DD_{r-2}\,$&$\DA_5$&$\DD_6$&$\DE_7$&$\DC_3$&$\,\,\DA_1\,\,$\\\hline
$\fg_0^{2,\ss}$&$\DB_{r-3}$&$\DD_{r-3}$&$\,\DA_2\times\DA_2\,$&$\DA_5$&$\DE_6$&$\DA_2$&$0$\\\hline
\end{tabular}
\end{table}

\subsection{Freudenthal varieties
}\label{ssec:Freudenthal}

We will now consider the adjoint variety $X_G$ and compute the fixed point components of the actions of the tori $H_2$ and $H_1$ introduced above (see Notation \ref{notn:downgradings}). In the cases in which $G$ is exceptional, the varieties that we will obtain are essentially those obtained in the Freudenthal magic square studied by Landsberg and Manivel in \cite{LM01}, see Table \ref{tab:Freudenthal}. In particular we will see that certain fixed point components for the torus $H_1$ coincide with the varieties of type ($2$) and ($3$) obtained in the classification of $\C^*$-actions of equalized bandwidth three actions with isolated sink and source (cf. Theorem \ref{thm:bw3}). These varieties appear in the literature as {\em Freudenthal varieties} (cf. \cite{KaYa, LM01}); in our setting they will be used to recognize adjoint varieties among contact manifolds  by means of equivariant K-theory, as we will see in Section \ref{contact_iso}.   

We start by noting that 
$$X_G^{H_2}=\bigsqcup_{m\in\bigtriangleup(S_2)\cap M_2} Y_m,\qquad X_G^{H_1}=\bigsqcup_{m\in\bigtriangleup(S_1)\cap M_1} Z_m,$$
where:
$$\begin{array}{l}
Y_m:=X_{G }\cap \P(\fg_m^{2}),\quad m\in\bigtriangleup(S_2)\cap M_2,\\[2pt]
Z_m:=X_{G }\cap \P(\fg_m^{1}),\quad m\in\bigtriangleup(S_1)\cap M_1.
\end{array}
$$

By construction, $Y_m$ is an isolated point when $m$ is a vertex of  $\bigtriangleup(S_2)$, and the same holds for the vertices of $\bigtriangleup(S_1)$. We will now study the fixed point components corresponding to inner lattices points. 

Let us first describe, for $i=1,2$, a subgroup $G_0^{i,\ss}\subset G$, $i=1,2$, whose Lie algebra is $\fg_{0}^{i,\ss}$. We will consider only the case $i=2$, being $i=1$ analogous; we follow the lines of argumentation of \cite[Section~2.3.1]{Tev}. We consider a $\Z$-basis $\{\mu_1,\mu_2\}$ of $M_2^\vee$ and, for each $j=1,2$, the derivation $D_j:\fg\to \fg$ defined by $D_j(x)=\mu_j(m)x$ for every $x\in \fg_m^2$, and for every $m\in M_2$. Since $\fg$ is semisimple, each $D_j$ is an inner derivation and there exists $\xi_j\in \fg^2_0$ such that $D_j=\ad_{\xi_j}$. Then $\fg^2_0$ can be described as $\{x\in\fg|\,\,[x,\xi_j]=0\,\,\mbox{for }j=1,2\}$, or as the Lie algebra of the subgroup:
$$
G^2_0:=\{g\in G|\,\,\Ad_g(\xi_j)=\xi_j\,\,\mbox{for }j=1,2\}^{\circ}\subset G.
$$
Then we may define $G^{2,\ss}_{0}$ as the commutator $[G^2_0,G^2_0]$.

\begin{lemma}\label{lem:innerRH}
For every $m\in M_2\cap\bigtriangleup(S_2)$ (resp. $m\in M_1\cap\bigtriangleup(S_1)$) the subvariety 
$Y_m\subset X_G$  (resp. $Z_m$) is a finite union of $G_0^{2,\ss}$-(resp. $G_0^{1,\ss}$-)  homogeneous varieties. 
\end{lemma} 

\begin{proof}
This is an adaptation of \cite[Theorem~2.6]{Tev} to our setting. We will show that $Y_m$ is a finite union of closed $G_0^2$-orbits, from which the statement follows; the case of $Z_m$ is analogous. By construction, the varieties $Y_m$ are $G_0^{2}$-invariant, so they are unions of $G_0^{2}$-orbits, and we only need to check that all these orbits are closed. To this end, we note first that for every $x\in Y_m$ we have: 
\begin{equation}\label{eq:tangents}
T_{X_G,x}\cap T_{\P(\fg^2_m),x}=T_{O_x,x},
\end{equation} 
where $O_x$ denotes the $G_0^{2}$-orbit of $x$. This follows by quotienting by $\langle x\rangle $ the equalities:
$$
\ad_x(\fg)\cap\fg^2_m=\ad_x\left(\bigoplus_{k\in\bigtriangleup(S_2)\cap M_2}\fg^2_k\right)\cap\fg^2_m=\bigoplus_{k}\ad_x(\fg^2_k)\cap\fg^2_m=\ad_x(\fg^2_0).
$$
Now if an orbit $O$ were not closed, its boundary would contain an orbit $O'$ of smaller dimension, and by the equality (\ref{eq:tangents}) one has that $T_{X_G}\cap T_{\P(\fg_m)}$ would have fibers of smaller dimension on the points of $O'$, contradicting semicontinuity. 
\end{proof}

Given a fixed point component $\Lambda$ of $X_{G }^{H_i}$, $i=1,2$, of weight $m$, we may consider the action of the torus $H/H_i$ on it. By Lemma \ref{lem:downgrading} this action has only isolated fixed points whose weights are the long roots of $G$ of $\Mo(H_i)$-degree equal to $m$. In particular, this shows that $\Lambda$ is the $G_0^{i,\ss}$-orbit of the class of an element $x\in \fg_\alpha$, where $\fg_\alpha\subset \fg^i_m$ and $\alpha$ is a long root of $G$. For instance, if $G$ is of type $\DF_4$ such a root does not exist in $\fg^2_0$, so $Y_0$ is empty in this case.  Moreover, this implies that the fixed point components of $X_G^{H_i}$ are the minimal orbits of the irreducible representations of $G^{i,\ss}_0$ with highest weights equal to one of the long roots of $G$. This allows us to study case by case the fixed locus of $X_G$ under the action of $H_i$. 

Table \ref{tab:Freudenthal} contains the description of the varieties $Y_m$, $Z_m$, for $m=0$ and for nonzero inner points $m\in\bigtriangleup(S_i)\cap M_i$:

\setlength{\tabcolsep}{2pt}
\begin{table}[h!]
\begin{tabular}{|c||c|c|c|c|c|}
\hline
 $\fg$& $\,X_G\,$&$Z_m$&$Z_0$&$Y_m$&$Y_0$\\\hline\hline
 $\DB_r$ 
 &$\DB_r(2)$&$\,\DA_1(1)\!\times\!\DB_{r-2}(1)\,$&$\,\DA_1(1)\sqcup\DB_{r-2}(2)\,$&$\star\sqcup\DB_{r-3}(1)$&$\DB_{r-3}(2)$\\\hline
 $\DD_r$ 
 &$\DD_r(2)$&$\,\DA_1(1)\!\times\!\DD_{r-2}(1)\,$&$\,\DA_1(1)\sqcup\DD_{r-2}(2)\,$&$\star\sqcup\DD_{r-3}(1)$&$\DD_{r-3}(2)$\\\hline
 $\DE_6$&$\DE_6(2)$&$\DA_5(3)$&$\DA_5(1,5)$&$\DA_2(2)\times\DA_2(1)$&$\,\DA_2(1,2)\!\sqcup \!\DA_2(1,2)\,$\\\hline
 $\DE_7$&$\DE_7(1)$&$\DD_6(6)$&$\DD_6(2)$&$\DA_5(4)$&$\,\DA_5(1,5)\,$\\\hline
 $\DE_8$&$\,\DE_8(8)\,$&$\DE_7(7)$&$\DE_7(1)$&$\DE_6(6)$&$\DE_6(2)$\\\hline
$\DF_4$&$\DF_4(1)$&$\DC_3(3)$&$\DC_3(1)$&$\,v_2(\DA_2(2))\,$&$\emptyset$\\\hline
$\DG_2$&$\DG_2(2)$&$v_3(\DA_1(1))$&$\emptyset$&$\star$&$\emptyset$\\\hline
\end{tabular}\vspace*{4mm}\caption{\label{tab:Freudenthal}Inner fixed point components corresponding to inner lattice points $0, m\in \bigtriangleup(S_i)\cap M_i$, $m\neq 0$. In the cases $\DB$ and $\DD$ the index $r$ is, respectively, bigger than or equal to $3$ and $4$. The symbol $v_n$ indicates the $n$-th Veronese embedding.}
\end{table}



\section{Contact manifolds with a compatible  torus action}\label{sec:contactaction}

In this section we collect some basic background on contact manifolds and torus actions on them (cf. \cite{KPSW}, \cite[$\S$4.1]{BWW}). 
 
\subsection{Contact manifolds}\label{ssec:contactmfd}

A {\em contact manifold} is a smooth projective variety $X$ of dimension $2n+1$ together with a line bundle $L$ (for short, we will say that the pair $(X,L)$ is a contact manifold) fitting in an exact sequence of vector bundles: 
\begin{equation}
0\ra{F}\lra{T_X}\stackrel{\sigma}{\lra}{L}\ra 0
\label{eq:contactstr}
\end{equation} 
such that the corresponding O'Neill tensor induces a skew-symmetric isomorphism
\begin{equation}
F\stackrel{\simeq}{\lra} F^\vee\otimes L.\label{eq:contactiso}
\end{equation}
The distribution $F\subset T_X$, the line bundle $L$ and the morphism $\sigma$ are called, respectively, the {\em contact distribution}, the {\em contact line bundle} and the {\em contact form} of $(X,L)$. Note that, combining (\ref{eq:contactstr}) and (\ref{eq:contactiso}) we may write
\begin{equation}\omega_X=L^{-(n+1)},\quad \det(F)=L^n.\label{eq:contact}\end{equation} 

\begin{remark}\label{rem:contactfano}
If a contact manifold $X$ is a Fano manifold, it is known (cf. \cite[Proposition~2.13]{KPSW}) that either $\Pic(X)\simeq \Z$, or $X\simeq \P(\Omega_{\P^{n+1}})$. Note also that Equation (\ref{eq:contact}) tells us that  $X$ is Fano if and only if  $L$ is ample. In the case in which $X\simeq \P(\Omega_{\P^{n+1}})$, we have that $L=\cO_{\P(\Omega_{\P^{n+1}})}(1)$; if $\Pic(X)\simeq \Z$, then $L$ is the ample generator of $\Pic(X)$, except for the case $X\simeq \P^{2n+1}$, where $L \simeq \cO_{\P^{2n+1}}(2)$.\end{remark}

The following Lemma allows us to identify $\HH^0(X,L)$ with the adjoint representation of the group of automorphisms of a Fano contact manifold $X$ different from $\P^{2n+1}$ and $\P(\Omega_{\P^{n+1}})$; we denote by $\CAut(X)\subset \Aut(X)$ the group of automorphisms of a contact manifold $X$ preserving the contact structure.

\begin{lemma}\label{lem:AutCAut}
Let $(X,L)$ be a contact Fano manifold with $\dim{X}=2n+1$, and $\Pic X=\ZZ L$. Then $\CAut(X)= \Aut(X)$. In particular $\Aut (X)$ extends to an action on $L$ so that the contact map $\sigma$ is $\Aut (X)$-equivariant, and we have $\Aut (X)$-equivariant isomorphisms $\Lie(\Aut (X))\simeq\HH^0(X,T_X)\simeq \HH^0(X,L)$.
\end{lemma}

\begin{proof}
Since $X\not \simeq \P^{2n+1}$ by Remark \ref{rem:contactfano}, applying  \cite[Sect. 2.3]{Ke01}, we know that $X$ is covered by a complete unsplit family of rational curves of degree $1$ with respect to $L$, called {\em contact lines}. By \cite[Theorem 4.4, Corollary 4.5]{Ke01},  the VMRT of this family spans $\P(F)$, and so the contact structure $F$ is unique. In particular every automorphism of $X$ preserves $F$, hence $\Aut(X)=\CAut(X)$. The second part of the statement follows by \cite[Proposition 1.1]{BeauvilleNilp}.
\end{proof}

Let us recall that a smooth subvariety $Y$ of a contact manifold $X$ as above is called {\em isotropic} if $T_Y\subset F_{|Y}$.  By definition, an isotropic subvariety of $X$ has dimension at most $n$; an isotropic subvariety of dimension exactly $n$ is called a {\em Legendrian} subvariety of $X$. 
For a Legendrian subvariety $Y \subset X$ we have a commutative diagram with short exact rows and columns:
$$
\xymatrix@C=35pt{T_{Y}\ar@{=}[r]\ar[d]&T_{Y}\ar[r]\ar[d]&0\ar[d]\\
F_{|Y}\ar[r]\ar[d]&(T_X)_{|Y}\ar[r]\ar[d]&L_{|Y}\ar@{=}[d]\\
\Omega_{Y_{\pm}}\otimes L_{|Y}\ar[r]&\cN_{Y|X}\ar[r]&L_{|Y}}
$$

The lower row of the diagram can be interpreted as follows:

\begin{proposition}\label{prop:Atiyah}
Let $(X,L)$ be a contact manifold with $\Pic X=\ZZ L$, and let $Y$ be a Legendrian subvariety.
Then the normal bundle $\cN_{Y|X}$  is isomorphic to the nontrivial extension of  $L_{|Y}$ by $\Omega_{Y} \otimes L_{|Y}$ corresponding to the Atiyah extension class $c_1(L_{|Y}) \in H^1(Y, \Omega_{Y})$. 
\end{proposition}

\begin{proof} The argument presented here belongs to an unpublished manuscript of the fourth author with J. Buczy\'nski, based on \cite[Section~2]{KPSW}. Let us denote by $\cL$ and $\cL_{Y}$ the Atiyah extensions corresponding to the line bundles $L$ and $L_{|Y}$ on $X$ and $Y$, respectively; they fit into the following commutative diagram of vector bundles on $Y$ with exact rows and columns: 
$$
\xymatrix@C=35pt{\cN^{\,\vee}_{Y|X}\ar@{=}[r]\ar[d]&\cN^{\,\vee}_{Y|X}\ar[r]\ar[d]&0\ar[d]\\
{\Omega_{X}}_{|Y}\ar[r]\ar[d]&\cL_{|Y}\ar[r]\ar[d]&\cO_{Y}\ar@{=}[d]\\
\Omega_{Y}\ar[r]&\cL_{Y}\ar[r]&\cO_{Y}}
$$
We use now the fact that $\cL^\vee$ admits a nondegenerate skew-symmetric form $\cL^\vee\otimes \cL^\vee\to L$ induced by the contact form in $X$, (cf. \cite[Section~2.2]{KPSW}). Moreover, the fact that $Y \subset X$  is Legendrian (Corollary \ref{cor:weightsBW2}) implies that the subbundle $\cN^{\,\vee}_{Y|X}\subset \cL_{|Y}$ is Lagrangian with respect to it. Hence the inner vertical exact sequence in the above diagram twisted  with $L_{|Y}$ and its dual fit in the following commutative diagram with exact rows: 
$$
\xymatrix@C=35pt{
\cN^{\,\vee}_{Y|X}\otimes L_{|Y} \ar[r] \ar[d] & \cL_{|Y}\otimes L_{|Y} \ar[r]\ar[d]^{\simeq} & \cL_Y\otimes L_{|Y} \ar[d]\\
\cL_Y^\vee \ar[r] & \cL_{|Y}^\vee \ar[r] & \cN_{Y|X}
}
$$
In particular, we get the  isomorphism $\cN_{Y|X}\simeq \cL_Y \otimes L_{|Y}$. Twisting with $ L_{|Y}$ the second row of the above diagram we conclude.
\end{proof}

\subsection{Compatible torus actions on contact manifolds}\label{ssec:compatible}

We will now consider contact manifolds $(X,L)$ for which there exists an action of a torus $H$ such that the contact form $\sigma$ is equivariant; we will say that such an action is {\em compatible with the contact structure}. In particular we will have an $H$-action on $F$, and the isomorphism (\ref{eq:contactiso}) will be $H$-equivariant.

The largest compatible torus action on a contact manifold $(X,L)$ is the one in which $H$ is a maximal torus of the identity component $ \CAut(X)^\circ$ of the group of contact automorphisms of $X$. In the case in which this group is reductive (for instance if $X$ is the twistor space of a compact quaternion-K\"ahler manifold), we have a precise description of the polytope of sections of this action, as a direct application of Lemma \ref{lem:AutCAut}: 

\begin{corollary}\label{cor:poly_sections}
Under the assumptions of Lemma \ref{lem:AutCAut}, assume that $\CAut(X)^\circ$ is reductive and let $H\subset \CAut(X)^\circ$ be a maximal torus. Then $\Gamma(X,L,H)$ is equal to the polytope $\bigtriangleup(\CAut(X)^\circ)$ generated by the roots of $\CAut(X)^\circ$.  
\end{corollary}

Let us now discuss some properties of the fixed point components of a compatible action of a complex torus $H$ on a contact manifold $(X,L)$, the weight map $\mu_L:X^H\to\Mo(H)$, and the compass at every component. The compatibility property tells us that, for every component $Y$, the weight $-\mu_L(Y)$ belongs to $\cC(Y,X,H)$. The rest of the elements of this compass are weights of the action of $H$ on $F_{|Y}$, which satisfy the following symmetry property (cf. \cite[Lemma 4.1]{BWW}), that follows from the fact that the contact isomorphism $F\simeq F^\vee\otimes L$ is $H$-equivariant. 

\begin{lemma}\label{rem:symweights}
Let $(X,L)$ be a contact manifold admitting a compatible action of a torus $H$, and let $Y\subset X^H$ be a fixed point component. For every weight $m$ of the action of $H$ on $F_{|Y}$ there exists another one $m'$ such that $m+m'=\mu_L(Y)$. 
\end{lemma}

Considering the restriction of the action to an extremal fixed point component, we obtain the following result, that was stated in \cite[Corollary~4.3]{BWW}:

\begin{lemma}\label{lem:sinksourceiso} 
Let $(X,L)$ be a contact manifold supporting a nontrivial compatible action of a torus $H$. If $Y$ is  an $H$-fixed point component such that $\mu_L(Y)\neq 0$, then $Y\subset X$ is an isotropic subvariety and the restriction of the contact structure $F_{|Y}$ contains $T_{Y}\oplus (\Omega_{Y}\otimes L)$ as a direct summand. Moreover the weight $-\mu_L(Y)$ appears with multiplicity $\dim Y+1$ in the compass $\cC(Y,X,H)$.
\end{lemma}

\begin{proof}
Denote by $d:=\mu_L(Y)\in\Mo(H)$ the weight of $L$ at every point of $Y$. Since $d\neq 0$ and $T_{Y}$ is at every point the eigenspace of ${T_X}_{|Y}$ associated to the weight zero, it follows that $T_Y$ is contained in the kernel of $\sigma_{|Y}:{T_X}_{|Y}\to L_{|Y}$, that is $F_{|Y}$. On the other hand, applying Lemma  \ref{rem:symweights} to the weights of $T_Y\subset F_{|Y}$, we conclude that $F_{|Y}$ contains an $H$-invariant vector subbundle $F'$ of rank equal to $\dim Y$, whose weights are all equal to $d$, so that $F'\oplus T_{Y}$ is a subbundle of ${T_X}_{|Y}$. The proof is finished observing that $F'$ is mapped isomorphically onto $\Omega_{Y}\otimes L$ via the composition of the isomorphism (\ref{eq:contactiso}) with the induced projection $F^\vee\otimes L\to \Omega_{Y}\otimes L$. For the second part note that, from the above arguments, the summand of weight $d$ in ${T_X}_{|Y}$ is an extension of $F'\simeq \Omega_{Y}\otimes L$ and $L_{|Y}$.
\end{proof}

We now focus on the case $H=\C^*$. The next statement shows that the central components of a compatible $\C^*$-action on $(X,L)$ inherit its contact structure:

\begin{lemma}\label{lem:legendrian2}
Let $(X,L)$ be a contact manifold supporting a nontrivial compatible $\C^*$-action. Then, for any irreducible component $Y_{0}$ such that $\mu_L(Y_0)=0$, the eigen-subbundle of weight zero $(F_{|Y_{0}})_{0}\subset F_{|Y_{0}}$ defines a contact form on $Y_{0}$. Moreover, the ranks $\nu^+(Y_0),\nu^-(Y_0)$ of the positive and negative parts of $\cN_{Y_0|X}$ are equal and, in particular, $\dim Y_{0}= \dim X-2\nu^+(Y_{0})$. 
\end{lemma}

\begin{proof}
Since $T_{Y_0}$ is the part of weight zero of $T_{X{|Y_0}}$, we have a weight decomposition (see Equation (\ref{eq:normal+-}) in Section \ref{ssec:actionC*}):
$$T_{X{|Y_0}}=\cN^-(Y_0)\oplus T_{Y_0}\oplus \cN^+(Y_0),$$
and the induced map $\sigma_{|T_{Y_0}}:T_{Y_0}\to L_{|Y_0}$ is surjective. Thus we have a decomposition of the contact distribution along $Y_0$: 
$$F_{|Y_0} \simeq  \cN^-(Y_0)  \oplus (F_{|Y_0} \cap T_{Y_0}) \oplus \cN^+(Y_0).$$
In particular, the contact isomorphism $F_{|Y_0} \simeq F^\vee_{|Y_0} \otimes L_{|Y_0}$ restricts to an isomorphism   
$F_{|Y_0} \cap T_{Y_0} \simeq (F_{|Y_0} \cap T_{Y_0})^\vee \otimes L_{|Y_0}$; this shows that  $(F_{|Y_0})_{0}=F_{|Y_0} \cap T_{Y_0}$ is a contact distribution on $Y_0$. The equality $\nu^+(Y_0)=\nu^-(Y_0)$ follows from the fact that the contact isomorphism sends $\cN^-(Y_0)$ isomorphically to $\cN^+(Y_0)^\vee \otimes L_{|Y_0}$.
\end{proof}

Furthermore, Equation (\ref{eq:contact}) immediately provides the following:

\begin{corollary}\label{cor:Y0contact}
Let $(X,L)$ be a Fano contact manifold supporting a nontrivial $\C^*$-action compatible with its contact structure. If $Y_0$ is an irreducible component such that $\mu_L(Y_0)=0$, then $(Y_0,L_{|Y_0})$ is a Fano contact manifold.  
\end{corollary} 

We finish this Section with two statements on contact varieties admitting a nontrivial compatible $\C^*$-action of bandwidth $2$, that will be used in Section \ref{contact_iso}. 

\begin{lemma}\label{lem:weightsBW2}
Let $(X,L)$ be a contact manifold supporting a nontrivial compatible $\C^*$-action of bandwidth two. Then the weights on the sink and the source of the action of $\C^*$ on $L$ are equal respectively to $-1$ and $1$. 
\end{lemma}

\begin{proof}
Let us denote by $Y_{-}$ the sink of the $\C^*$-action on $L$, and by $Y_{+}$ the source.
Denote by $d_\pm$  the weights of  $L$ on $Y_\pm$.
At every point of $Y_{-}$, the weights of the action on $(T_{X})_{|Y_{-}}$ are all non positive, therefore, by (\ref{eq:contact}), $d_-<0$. In a similar way we can prove that $d_+ >0$. By the assumption on the bandwidth we have $d_+=d_-+2$, and the result follows.  
\end{proof}

 Combining  Lemma \ref{lem:sinksourceiso}  and Lemma \ref{lem:weightsBW2}, we get the following: 

\begin{corollary}\label{cor:weightsBW2}
Let $(X,L)$ be a contact manifold of dimension $2n+1$, supporting a nontrivial compatible $\C^*$-action  of bandwidth two. Then the extremal fixed components $Y_\pm$ are Legendrian subvarieties  of $X$, $F_{|Y_{\pm}}$ is isomorphic to $T_{Y_{\pm}}\oplus (\Omega_{Y_{\pm}}\otimes L)$ and, in particular, $\dim Y_{\pm}=n$. 
\end{corollary}

\begin{proof}
Let us deal with $Y_-$, being the case of $Y_{+}$ completely analogous. Since the weights of the action on $F_{|Y_-}$ are all non positive and, by Lemma \ref{rem:symweights}, symmetric with respect to $-1/2$, it follows that the only weights are $0$ and $-1$, appearing both with multiplicity $n$. We then conclude arguing as in Lemma \ref{lem:sinksourceiso}. 
\end{proof}



\section{Actions on contact manifolds with isolated extremal fixed points}
\label{contact_iso}

Throughout this section, given a contact manifold $(X,L)$, we will denote by $\CAut(X)^\circ$ the identity component of the group of contact 
automorphisms of $X$. For a simple algebraic group $G$  we denote by $X_G$ the closed orbit in the
adjoint representation of $G$ and by $L_G$ the pull-back of $\cO(1)$ via the
embedding $X_G\hookrightarrow\PP(\g)$ where $\g$ is the Lie algebra of $G$. The goal of this section is to prove the following:

\begin{theorem} \label{thm:LB_isolated_pts}
Let $(X,L)$ be a contact Fano manifold with $\dim{X}=2n+1$, and $\Pic X=\ZZ L$. Suppose that $G=\CAut(X)^\circ$ is reductive of rank $r\geq 2$, and that the action of the maximal torus $H\subset G$ on $X$ has isolated points as extremal fixed point components. Then $G$ is simple of one of the following types: 
$$
\mbox{$\DB_r$ $(r\geq 3)$,  $\DD_r$ $(r\geq 4)$,  $\DE_r$ $(r=6,7,8)$,
$\DF_4$, $\DG_2$,}
$$ and $(X,L)=(X_G,L_G)$.
\end{theorem}

\begin{remark}\label{rem:red_dim}
Note that the statement is known for $n\leq 4$ without any assumptions on the rank of $G$ and on $X^H$ (see \cite[Section~1.1]{BWW} and the references therein). In the proof of the Theorem we will then assume that $n\geq 5$. Note also that for the cases $n=5,6$, the statement has been proved in \cite[Theorem 6.2]{RW} without the assumption on the extremal fixed points. 
\end{remark}

\textbf{Outline and methods of the proof of Theorem \ref{thm:LB_isolated_pts}.}
We achieve the proof of Theorem  \ref{thm:LB_isolated_pts} in five steps. In Step I we prove the simplicity of $G$. To this end, in Lemma \ref{lem:polyroot} we first observe that the polytope of sections and of fixed points coincide and that they are also equal to $\bigtriangleup(G)$. Then in Corollary \ref{cor:LSisolated} we deduce the simplicity of $G$, so that we know all the possibilities for $\bigtriangleup(X,L,H)$, and we eliminate the cases in which $G$ is of type $\DC$ or of type $\DA_r$ with $r \ge 3$ (see Lemma \ref{lem:noCtype}). Step II and III contain the key point which allows us to extend the previous results of \cite{BWW, RW} to any dimension of the contact variety. In these steps we will make use of classification results of bandwith three and two varieties which have been recently obtained in \cite{ORSW}. These varieties arise as submanifolds of $X$ (see Propositions \ref{prop:Xmu} and \ref{prop:Z}) and play a central role to determine the combinatorial data which will be needed. In Step II we use Lemma \ref{lem:reduction} to reduce the action of $G$ on $(X,L)$ to the action of a subgroup $S_2\subset G$ isomorphic to $\SL(3)$. Then we focus on the induced action of a rank two torus $H_2\subset S_2$ on $(X,L)$, and in Step III we describe all the fixed point components and compasses with respect to the action of this smaller torus. As a consequence, in Step IV, we prove that such combinatorial data coincide with the fixed point components and compasses of the action of $H_2$ on the corresponding adjoint variety $(X_G, L_G)$ (cf. Corollaries \ref{cor:H2griddata} and \ref{cor:XGXG'}). All the information collected in these steps about the $H_2$-action allows us to conclude in Step V that the fixed point components and the compasses of the action of the maximal torus $H$ on $(X,L)$ and on $(X_G,L_G)$ are the same (see Lemmas \ref{lem:grid1} and \ref{lem:grid2}). Applying Proposition \ref{prop:BWW224} we obtain that $(X,L)\simeq (X_G,L_G)$ as stated. 

\subsection{Step I: Simplicity and type of the automorphism group}

The following statement, whose hypotheses are obviously fulfilled under the assumptions of Theorem \ref{thm:LB_isolated_pts}, provides the equality of the polytopes of fixed points and of sections of the action of the maximal torus $H\subset G=\CAut(X)^\circ$ on $(X,L)$. 

\begin{lemma}\label{lem:polyroot}
Let $X$ be a contact Fano manifold with $\dim{X}=2n+1$ and $\Pic X=\ZZ L$, and $H$ be a maximal torus of $G=\CAut(X)^\circ$. Assume that for every extremal fixed point component $Y\subset X^H$ we have $\HH^0(Y,L_{|Y})\neq 0$. Then $$\bigtriangleup(X,L,H) = \Gamma(X,L,H) =\bigtriangleup(G).$$  
\end{lemma}

\begin{proof}
The equality $\Gamma(X,L,H) =\bigtriangleup(G)$ follows by Corollary \ref{cor:poly_sections}; we will prove now that $\Gamma(X,L,H)= \bigtriangleup(X,L,H)$. 
The inclusion $\Gamma(X,L,H)\subseteq \bigtriangleup(X,L,H)$ in the case in which $L$ is ample is a general fact, (cf. \cite[Lemma 2.4(2)]{BWW}). 
To prove the other inclusion, we proceed as in the proof of \cite[Corollary 3.8]{BWW}. Let us take an extremal fixed point component $Y$, with weight $\mu_{L}(Y)\in \bigtriangleup(X,L,H)$; by hypothesis  $\HH^0(Y,L_{|Y})\ne 0$.  Applying \cite[Lemma 3.6]{BWW} the $H$-equivariant map $\HH^0(X,L)\to \HH^0(Y,L_{|Y})$ is surjective and gives an isomorphism of $\HH^0(Y,L_{|Y})$ with the eigenspace of $\HH^0(X,L)$  corresponding to the eigenvalue $\mu_L(Y)$, hence we deduce that $\mu_{L}(Y)\in \Gamma(X,L,H)$.
\end{proof}

Applying the machinery developed in \cite{BWW} we achieve the goal of this step: 
\begin{corollary}\label{cor:LSisolated}
Under the assumptions of Theorem \ref{thm:LB_isolated_pts},  $G$ is simple.\end{corollary}

\begin{proof}
Since $\HH^0(Y,L_{|Y})\simeq \C$ for every extremal fixed point component $Y\subset X^H$, Lemma \ref{lem:polyroot} provides the equality $\bigtriangleup(X,L,H) = \Gamma(X,L,H)$. Hence $G$ is a semisimple group by \cite[Lemma 4.6]{BWW}, and arguing as in the proof of \cite[Proposition 4.8]{BWW} we deduce the simplicity of $G$. 
\end{proof}

\begin{remark}\label{rem:scG}
If $G'$ is a simply connected covering of $G$, and $H' \subset G'$ is a maximal torus dominating $H$, we have equalities
$$\Mo(H')\otimes_\Z\R= \Mo(H)\otimes_\Z\R, $$
$$ \bigtriangleup(X,L,H')=\bigtriangleup(X,L,H)=\Gamma(X,L,H)=\Gamma(X,L,H'),$$
therefore we will assume from now on that $G$ is simply connected.
 \end{remark}

\begin{lemma} \label{lem:noCtype}
Under the assumptions of Theorem \ref{thm:LB_isolated_pts}, the group $G$ cannot be of type $\DC$ or of type $\DA_r$ with $r \ge 3$.
\end{lemma}

\begin{proof}
Using Remark \ref{rem:scG}, Corollary \ref{cor:LSisolated}, and Lemmas \ref{lem:AutCAut}, \ref{lem:polyroot}, one has that \cite[Assumptions 5.1, 5.2]{BWW} are satisfied, so the result follows from \cite[Proposition 5.9 and Lemma 5.10]{BWW}.\end{proof}

\subsection{Step II: Downgradings to subtori of rank two and one}\label{ssec:StepII}

In this subsection we consider the restrictions of the action of $G$ on $(X,L)$ to a subgroup $S_2\subset G$ isomorphic to $\SL(3)$ satisfying the requirements of Lemma \ref{lem:reduction}, and then to certain subgroups of $S_2$ isomorphic to $\SL(2)$ as described in Section \ref{ssec:downgradings} (see Notation \ref{notn:downgradings}). We will consider one of these subgroups for each of the roots of $S_2$, so we will start by introducing some notation (see Figure \ref{fig:hexagon2}).

\begin{notation}\label{notn:donwgradings2}
Denoting by $\alpha_0,\dots,\alpha_5$, $\alpha_6:=\alpha_0\in\Mo(H_2)$ the roots of $S_2$ ordered counterclockwise,
and setting $\beta_i=(\alpha_i+\alpha_{i+1})/3$, $i=0,\dots,5$, the lattice 
$\Mo(H_2)$ is generated by $\alpha_0$ and $\beta_0$. In the sequel, the indices of $\alpha$'s and $\beta$'s are between $0$ and $5$ and by convention they are taken modulo $6$. For every $i= 0,\dots,5,$ let $H^i_1\subset H_2$ be the $1$-dimensional subtorus corresponding to the orthogonal projection $\pi_i^*:\Mo(H_2)\to\Mo(H_1^i)\simeq \Z$ sending $\beta_i$ to zero. As in Section \ref{ssec:downgradings}, each projection determines a subgroup $ S_1^i\subset S_2$ isomorphic to $\SL(2)$, whose Lie algebra contains $\fg_{\alpha_{i-1}}$. The lattice points contained in the root polytope  $\bigtriangleup(S_2)$ are precisely the points of the set $$\Hex:=\{0, \alpha_i,\beta_i|\,\,\,i=0,\dots,5\},$$ and the lattice points in the root polytope  $\bigtriangleup(S^i_1)$ are the integers $\{-2,-1,0,1,2\}$. 
\end{notation}

\begin{figure}[h!]
\includegraphics[height=5cm]{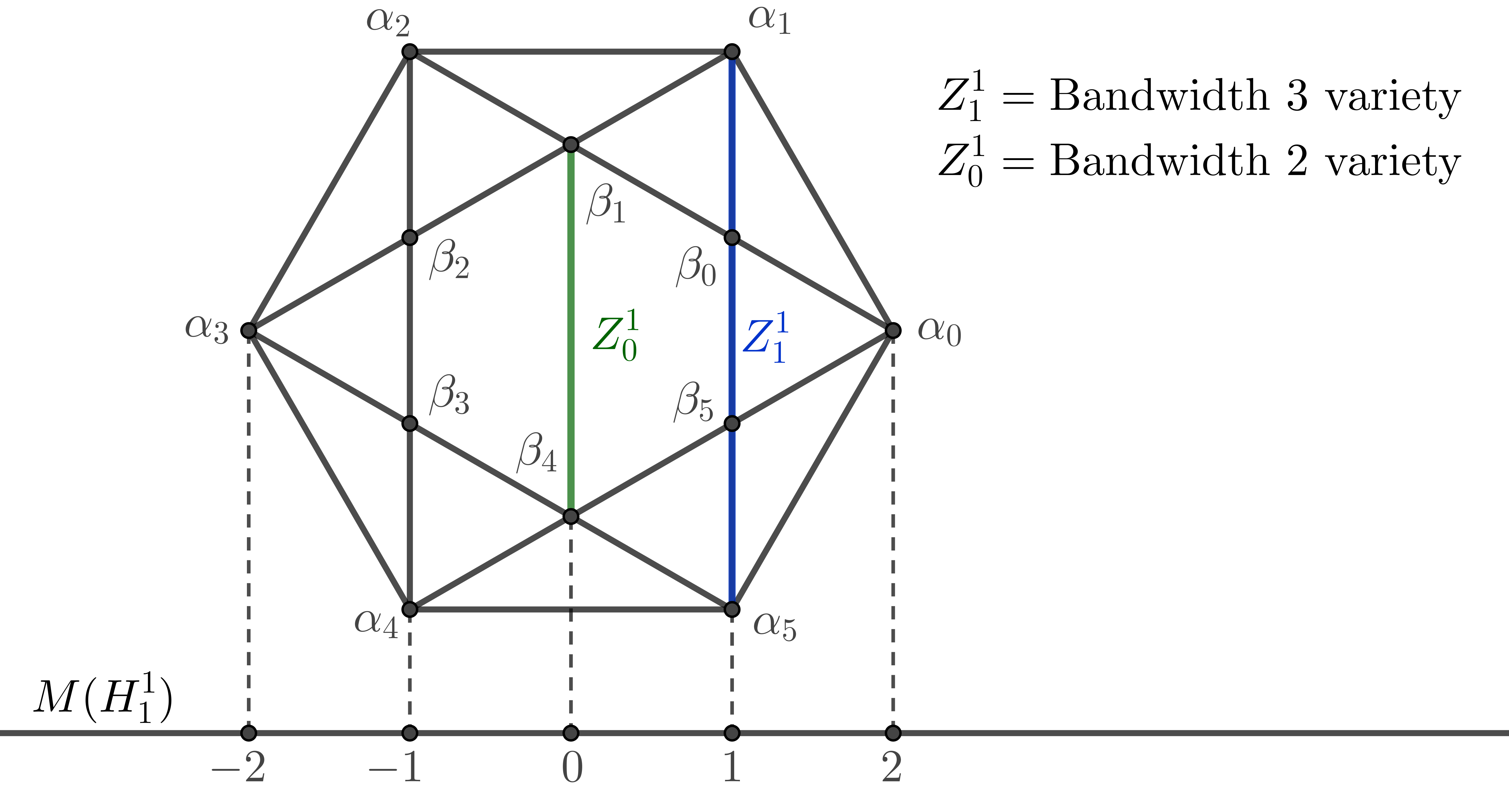}
\caption{$H_2$-weights and downgrading to $H^1_1$.}
\label{fig:hexagon2}
\end{figure}

In Step III we will describe completely the possible isomorphism classes of the fixed point components of the action of $H_2$. Here we will only introduce some notation about them, and their most basic properties. We note first that the weights of the induced action of $H_2$ and $H^i_1$  on $L$ are elements of $\Hex$ and of the set $\{-2,-1,0,1,2\}$, respectively. In fact, by Lemma \ref{lem:polyroot} and Lemma \ref{lem:downgrading}: 

\begin{lemma}\label{lem:polyroot2}
The polytope of fixed points $\bigtriangleup(X,L,H_2)$ is equal to $\bigtriangleup(S_2)$.  
\end{lemma} 
 
\begin{remark}\label{rem:isolatedpoints} 
Using Lemma \ref{lem:reduction} the inverse image in $\bigtriangleup(G)$ of every vertex of $\bigtriangleup(S_2)$ (resp.  $\bigtriangleup(S_1)$) is unique. Then  the basic properties of downgradings (see Lemma \ref{lem:downgrading} (iii,v)) tell us that the extremal fixed components of the action of $H_2$ (resp.  $H_1$) on $X$ are isolated points. 
\end{remark}

\begin{notation}\label{not:fixedH2}
In the sequel, the extremal $H_2$-fixed points associated to the weights $\alpha_i$ will be denoted by $y_i$, the fixed point components associated to weights $\beta_i$ will be denoted by $Y_{\beta_i,k}$, and  the fixed point components associated to the weight zero will be denoted by $Y_{0,k}$  (for $k$ in a finite set of indices). Moreover we set $Y_{\beta_i}:=\bigsqcup_k Y_{\beta_i,k}$, for $i=0,\dots,5$, $Y_0:=\bigsqcup_{k} Y_{0,k}$. 
\end{notation}

\begin{notation}\label{not:Z} For every $i=0,\dots, 5$, the irreducible fixed point components of the $H_1^i$-action will be denoted by 
$$
Z^i_{-2}=\{y_{i+2}\},\,\,Z^i_{-1,k},\,\,Z^i_{0,k},\,\,Z^i_{1,k},\,\,Z^i_{2}=\{y_{i-1}\},
$$
where the first subindex indicates the $H_1^i$-weight of $L$, and $k$ belongs to a  finite set of indices for every weight. The next statement shows that the components of weights $\pm 1$ are unique, so we will simply write $Z_{\pm 1}^i:=Z^i_{\pm 1,k}$. In particular $Z_{1}^i$ will contain all the irreducible $H_2$-fixed components of weights $\alpha_i,\beta_{i-1},\beta_{i-2},\alpha_{i-2}$. Moreover for every component $Y_{\beta_i,k}$ there exists a unique $H_1^i$-fixed component of weight zero containing it; however, it is not true a priori that given a $H_1^i$-fixed component of weight zero $Z^i_{0,k}$ there exists an $H_2$-fixed component of weight $\beta_i$ contained in it. In any case, we may adjust the indices $k$ so that we may write $Y_{\beta_i,k}\subset Z^i_{0,k}$. 
\end{notation}

\begin{lemma}\label{lem:defXmu}
For every $i=0,\dots, 5$ there exists a unique irreducible fixed point component $Z_1^i$ of the action of $H^i_1$ on $(X,L)$ associated to the weight $1$.
\end{lemma}
\begin{proof}
We will do the proof in the case $i=1$. We note first that, arguing as in Remark \ref{rem:isolatedpoints}, $y_{0}$ is the only isolated $H^1_1$-fixed point associated to the weight $2$. Then applying Theorem \ref{thm:BB_decomposition}, the dimension of $H_2(X,\Z)$ equals the number of irreducible $H^1_1$-fixed point components associated to the weight $1$. Since $\Pic(X)\simeq\Z$, we conclude that there is a unique irreducible component $Z_1^1$ of this kind.
\end{proof}

We end this step by studying the induced action of $H_2/H^i_1$ on $Z_{1}^i$ and $Z_{0,k}^i$: 

\begin{lemma}\label{lem:XmuZ}
For every $i,k$, being $Z_{0,k}^i$ an $H_1^i$-fixed component containing an $H_2$-fixed component $Y_{\beta_i,k}$, the torus $H_2/H^i_1$ acts on $(Z_{1}^i,L_{|Z_{1}^i})$ and $(Z_{0,k}^i,L_{|Z_{0,k}^i})$ with bandwidth equal to three and two, respectively.
\end{lemma}

\begin{proof} 
As in Notation \ref{not:Z}, the extremal fixed points of the $H_2/H^i_1$-action have weights $\alpha_i$ and $\alpha_{i-2}$, thus the bandwidth of the action is three. In the case of $Z_{0,k}^i$ we note first that $Y_{\beta_i,k}\subset Z_{0,k}^i$ is, by construction, an extremal fixed point of the  $H_2/H^i_1$-action on $Z_{0,k}^i$. Applying Lemma \ref{lem:sinksourceiso} it follows that $-\beta_i$ appears with multiplicity $\dim Y_{\beta_i,k}+1\geq 1$ in $\cC(Y_{\beta_i,k},X,H_2)\cap \ker(\pi_i^*)=\cC(Y_{\beta_i,k},Z_{0,k}^i,H_2/H_1^i)$ (see Lemma \ref{lem:compass_prop}), so $Y_{\beta_i,k}\subsetneq Z_{0,k}^i$. Then the other extremal component of the $H_2/H^i_1$-action on $Z_{0,k}^i$ must be an $H_2$-fixed component  of weight either zero or $-\beta_i=\beta_{i+3}$. The first case is not possible by Lemma \ref{lem:legendrian2}, so we conclude that the other extremal component is associated to the weight $\beta_{i+3}$, and that the bandwidth of the action is two.
\end{proof}


\subsection{Step III: Computing the fixed components and compasses of the action of $H_2$}\label{ssec:StepIII}

We start with a statement collecting properties of compasses at extremal fixed point components. Keeping in mind that $\bigtriangleup(X,H_2,L)=\bigtriangleup(S_2)$ (cf. Lemma \ref{lem:polyroot2}), the result follows by applying \cite[Corollary 5.6]{BWW} to our case.   
\begin{lemma} \label{lem:compass_extremalpoints} 
Let $\alpha\in \bigtriangleup(S_2)$ be a vertex, $y_{\alpha}\in X^{H_2}$ be the corresponding fixed point, $\delta$ be an edge of $\bigtriangleup(S_2)$ containing $\alpha$, and write $\delta\cap\Mo(H_2)=\{\alpha,\alpha'\}$. Then:
\begin{itemize}
\item [(1)] The compass $\cC(y_{\alpha},X,H_2)$ contains ${\alpha}^{\prime}-\alpha$ with multiplicity one.
\item [(2)] $\cC(y_{\alpha},X,H_2)$ contains $-\alpha$ with multiplicity one.
\item [(3)] Let $\tau\in \Mo(H_2)\otimes_\Z\R$ be the convex cone generated by the shift $\bigtriangleup{( S_2)}- \alpha$. Then $\cC(y_{\alpha},X,H_2)\subseteq \tau \cap (\alpha-\tau)\cap \Mo(H_2)$. 
\end{itemize}
\end{lemma}
\begin{proof} 

It is enough to note that the action of $S_2$ on $(X,L)$ satisfies \cite[Assumptions 5.1]{BWW}, so that the proof follows from \cite[Corollary 5.6]{BWW}; in fact the only nontrivial assumptions to be checked in that list are (4) and (5), and they hold by Lemma \ref{lem:polyroot2} and Remark \ref{rem:isolatedpoints}, respectively.  
\end{proof}

The following lemma, that describes the compass of the $H_2$-action on $X$ at an extremal fixed point, is a generalization of \cite[Lemma 5.15]{BWW}, where the statement has been proved for $n=3,4$. 

\begin{proposition}\label{prop:compassHex} 
Let us consider the action of $H_2$ on $(X,L)$. The compasses at the fixed point components which are associated to  weights different from zero are:
$$\mathcal{C}(y_{i}, X, H_2)= \{\alpha_{i+1}-\alpha_i, \alpha_{i-1}-\alpha_i, -\alpha_i, (\beta_i-\alpha_i)^{n-1},(\beta_{i-1}-\alpha_i)^{n-1} \}$$
$$\mathcal{C}(Y_{\beta_i,k}, X, H_2)= \{(\beta_{i-1}-\beta_i)^{n-2-\dim Y_{\beta_i,k}}, (\beta_{i+1}-\beta_i)^{n-2-\dim Y_{\beta_i,k}},\beta_{i-2}-\beta_i, $$ 
$$\qquad \alpha_{i+1}-\beta_i,\beta_{i+2}-\beta_i,\alpha_{i}-\beta_i,-\beta_i^{\dim Y_{\beta_i,k}+1}\}$$
\end{proposition}

\begin{figure}[h!]
\includegraphics[height=5cm]{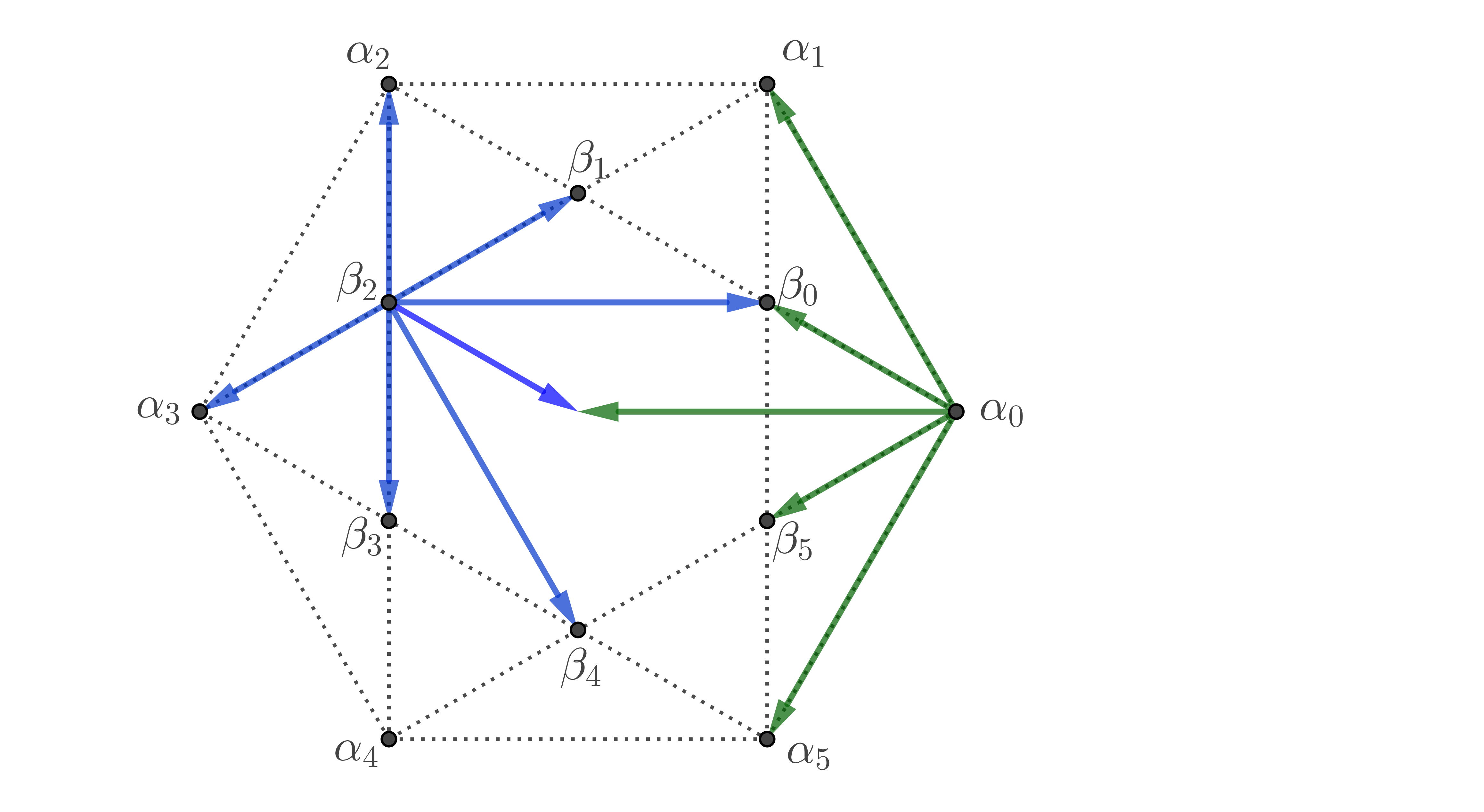}
\caption{Compasses at fixed components $y_{0}$ and $Y_{\beta_2}$.}
\label{fig:Compass1}
\end{figure}

\begin{proof}
By Lemma \ref{lem:compass_extremalpoints} (1,2) the elements $\alpha_{i+1}-\alpha_{i},\alpha_{i-1}-\alpha_{i},-\alpha_{i}$ belong to the compass $\mathcal{C}(y_{i}, X, H_2)$ with multiplicity one. Using item (3) of the same Lemma, $\beta_{i}-\alpha_{i}$, $\beta_{i-1}-\alpha_{i}$ are the only other possible elements in the compass, and by Lemma \ref{rem:symweights} they occur with equal multiplicity. The extremal fixed components are points (Remark \ref{rem:isolatedpoints}), hence the number of elements in $\mathcal{C}(y_{i}, X, H_2)$ is equal to  $\dim X$. The multiplicity of $\beta_{i}-\alpha_{i}$, $\beta_{i-1}-\alpha_{i}$  in the compass must then be equal to $n-1$ and the first claim follows. 

We now describe the compass at a component $Y_{\beta_i,k}$. To this end, consider the one dimensional subtorus $H_1^i$. Taking the  component $Z_1^{i+1}$ containing it (see Lemma \ref{lem:defXmu}), and applying Lemma \ref{lem:compass_prop} we get: 
$$\pi_i^*(\mathcal{C}(Y_{\beta_i,k},X,H_2))=\mathcal{C}(Z_1^{i+1},X,H_1^i)=\pi_i^*(\mathcal{C}(y_{i+1},X,H_2))=\{-2,{-1}^{n},1\}$$ 
where the last equality is gotten using the first part of this statement. By Lemma \ref{lem:compass_prop2}, the only element of $\mathcal{C}(Y_{\beta_i,k},X,H_2)$ that can be projected to $1$ is $\alpha_{i}-\beta_{i}$, then this element has multiplicity one in $\mathcal{C}(Y_{\beta_i,k},X,H_2)$. The same method with a different choice of the one dimensional subtorus shows that also $\alpha_{i+1}-\beta_{i}$  occurs with multiplicity one. Using the symmetry given by Lemma \ref{rem:symweights} we obtain that $\beta_{i+2}-\beta_{i}$, $\beta_{i-2}-\beta_{i}$ have multiplicity one as well. Moreover, applying Lemma \ref{lem:sinksourceiso} we know that the element $-\beta_i$ occurs with multiplicity $\dim Y_{\beta_i,k}+1$. 

Finally, since the number of elements of $\mathcal{C}(Y_{\beta_i,k},X, H_2)$, counted with  multiplicity, must be equal to $\dim{X}-\dim{Y_{\beta_i,k}}$, and using again Lemma \ref{rem:symweights}, it follows that the elements $\beta_{i-1}-\beta_{i}$ and $\beta_{i+1}-\beta_{i}$ have both multiplicity equal to $n-2-\dim{Y_{\beta_i,k}}$. 
\end{proof}

Now we determine the fixed point components of the $H_2$ action on $(X,L)$, by analyzing first  in detail the bandwidth three varieties $Z_1^i$ (see Lemma \ref{lem:defXmu}).

\begin{proposition}\label{prop:Xmu}
Let $Y_{\beta_{i-1}}\subset Z_1^i$ be as in Notation \ref{not:Z} and \ref{not:fixedH2}. Then $\dim Z_1^i = n-1$, the pair $(Z_1^i,Y_{\beta_{i-1}})$ is one of the following:
\begin{enumerate}
\item $(\P^1 \times \DB_{(n-1)/2}(1), ~ \star \sqcup \DB_{(n-3)/2}(1))$,
\item $(\P^1 \times \DD_{n/2}(1), ~ \star \sqcup \DD_{(n-2)/2}(1))$,
\item $(\DC_3(3), \DA_2(2))$,
\item $(\DA_5(3), \DA_2(2) \times  \DA_2(1))$,
\item $(\DD_6(6), \DA_5(4))$,
\item $(\DE_7(7),\DE_6(6))$,  
\end{enumerate}
and the restriction of $L$ to every positive dimensional component $Y_{\beta_{i-1},k}\subseteq Y_{\beta_{i-1}}$ is the ample generator of $\Pic(Y_{\beta_{i+1},k})$, except in case (3), where $L_{|Y_{\beta_{i-1},k}} \simeq \cO_{\P^2}(2)$.
\end{proposition}

\begin{proof}
As observed in Lemma \ref{lem:XmuZ}, we have a bandwidth three action of $H_2/H_1^i\simeq \C^{*}$ on $Z_1^i$; we claim first that this action is equalized. In fact, given any fixed point component $Y\subset Z_1^i$, one has $\mathcal{C}(Y, Z_1^i, H_2/H_1^i)= \mathcal{C}(Y, X, H_2)\cap \ker{\pi^*_i}$ by Lemma \ref{lem:compass_prop}. Using the description of $\mathcal{C}(Y, X, H_2)$ given in  Proposition \ref{prop:compassHex} for every fixed component $Y\subset Z_1^i$, this shows that the only possible elements of $\mathcal{C}(Y, Z_1^i, H_2/H_1^i)$ are $\pm 1$. Note also that $Z_1^i$ has isolated extremal fixed points $y_{i}$, $y_{i-2}$, hence $\dim Z_1^i$ is equal to the number of elements of $\mathcal{C}(y_i, X, H_2)\cap \ker{\pi^*_i}$, which is $n-1$ by the same Proposition.

Since for the proof of Theorem \ref{thm:LB_isolated_pts} we have assumed that $n\geq 5$ (see Remark \ref{rem:red_dim}), we may apply Theorem \ref{thm:bw3} to $(Z_1^i,L_{|Z_1^i})$ to obtain the description of the fixed point components $Y_{\beta_{i-1},k}$ and of $L_{|Y_{\beta_{i-1},k}}$.  
The fact that $Z_1^i$ cannot be of type (1) in Theorem \ref{thm:bw3} when $\dim Z_1^i \geq 3$ 
has been proved in \cite[Corollary~6.7]{RW}.
\end{proof}

\begin{remark}\label{rem:allequal}
Theorem \ref{thm:bw3} tells us also that $Y_{\beta_{i-1}}\simeq Y_{\beta_{i-2}}$, for every $i$. Then, since two consecutive components $Z^i_1$, $Z^{i+1}_1$ contain $Y_{\beta_{i-1}}$ (cf. Notation \ref{not:fixedH2}) and the isomorphism class of $Y_{\beta_{i-1}}$ determines $Z^i_1$ and $Z^{i+1}_1$, it follows that the varieties $Z^i_1$ are isomorphic for all $i$, and the same holds for the varieties $Y_{\beta_i}$. Note also that the list of pairs $(Z_1^i,Y_{\beta_{i-1}})$ of Proposition \ref{prop:Xmu} coincides with the list of pairs $(Z_1^i,Y_{\beta_{i-1}})$ obtained for adjoint varieties (see Table \ref{tab:Freudenthal}).
\end{remark}

\begin{proposition}\label{prop:Z} 
Let $Y_{\beta_i,k}\subset Z_{0,k}^i$, $Y_0$ be as in  Notations \ref{not:Z} and \ref{not:fixedH2}. Then  the triple $(Z_{0,k}^i,Y_{\beta_i,k},Z_{0,k}^i \cap Y_{0})$ is one of the following:
\begin{itemize}
\item[(0)] $(\P^1, \star,\emptyset)$;
\item[(1)] $(\DB_{(n-1)/2}(2), \DB_{(n-3)/2}(1),\DB_{(n-3)/2}(2))$;
\item[(2)] $(\DD_{n/2}(2),  \DD_{(n-2)/2}(1),\DD_{(n-2)/2}(2))$;
\item[(3)]  $(\DC_3(1),\DA_2(2),\emptyset)$;
\item[(4)]  $(\DA_5(1,5), \DA_2(2) \times  \DA_2(1), \DA_2(1,2) \sqcup  \DA_2(1,2))$; 
\item[(5)]  $(\DD_6(2), \DA_5(4),\DA_5(1,5))$;
\item[(6)]  $(\DE_7(1),\DE_6(6),\DE_6(2))$.
\end{itemize}
Moreover, every positive dimensional component $Y_{0,r}\subseteq Z_{0,k}^i \cap Y_{0}$ is a contact manifold with contact line bundle $L_{|Y_{0,r}}$. 
\end{proposition}

\begin{proof} Note that, as observed in Lemma \ref{lem:XmuZ} the action of $H_2/H_1^i\simeq \C^*$ on  $(Z_{0,k}^i,L_{|Z_{0,k}^i})$ has bandwidth two. By Proposition \ref{prop:compassHex} and Lemma \ref{lem:compass_prop} (ii) this action is equalized at the sink and the source, hence it is equalized by \cite[Lemma 5.8]{ORSW}. Moreover, by Corollary \ref{cor:Y0contact}, $Z_{0,k}^i$ and subsequently every positive dimensional irreducible component of $Z_{0,k}^i \cap Y_{0}$ is a contact manifold whose contact line bundle is the restriction of $L$.

By Remark \ref{rem:allequal} and Corollary \ref{cor:weightsBW2}  the extremal components $Y_{\beta_i,k}, Y_{\beta_{i+3},k}$ of the $H_2/H_1^i$ action on $Z_{0,k}^i$ are isomorphic to each other, and  isomorphic to one of the connected components $Y$ of the varieties $Y_{\beta_i}$ appearing in Proposition \ref{prop:Xmu}. Moreover, again by Corollary \ref{cor:weightsBW2}, we get $\dim Z_{0,k}^i = 2 \dim Y+1$.

When $Y$ is a point one has  $\dim Z_{0,k}^i=1$, and we obviously have $Z_{0,k}^i \simeq \P^1$. If  $Y \simeq \DA_2(2)$ we have  $\dim Z_{0,k}^i=5$, and $L_{|\DA_2(2)}\simeq \cO_{\P^2}(2)$ by Proposition \ref{prop:Xmu}; using \cite[Lemma~2.9~(i)]{ORSW}, $L$ is then the second power of the ample generator of $\Pic(Z_{0,k}^i)$, therefore the index of $Z_{0,k}^i$ is six and so $Z_{0,k}^i  \simeq \P^5$. Note that a  $\C^*$-action on $\P^5$ with two fixed disjoint $\P^2$'s does not have other fixed point components. If  $Y \simeq  \DA_2(2) \times  \DA_2(1)$ then $\dim Z_{0,k}^i=9$, and the Picard number of $Z_{0,k}^i$ must be larger than one (see \cite[Lemma~2.9(1)]{ORSW}). Then we conclude that $Z_{0,k}^i\simeq \DA_5(1,5)$ (see Remark \ref{rem:contactfano}). In this case the bandwidth two $\C^*$-action has been described in \cite[Example 5.18]{ORSW}, where it is shown that the central component is  $\DA_2(1,2) \sqcup  \DA_2(1,2)$.

In all the other cases, since the normal bundles $\cN_{Y_{\beta_i,k}|Z^i_{0,k}}$, $\cN_{Y_{\beta_{i+3},k}|Z^i_{0,k}}$ are uniquely determined by $Y$ and $L$, and their duals twisted with $L$ are globally generated (Proposition \ref{prop:Atiyah}), we can use Theorem \ref{thm:uniqueBW2} to conclude that the variety $Z^i_{0,k}$ is also determined by these data. In particular all the $Z^i_{0,k}$ are isomorphic to the corresponding $H_2$-fixed components obtained in the case in which $X$ is an adjoint variety  (see Table \ref{tab:Freudenthal}), so we obtain the list of triples in the statement.
\end{proof}

\begin{remark}\label{rem:allequal2}
Note that the case (0) of Proposition \ref{prop:Z} appears when $X$ contains fixed components of the types (1,2) of the list of Proposition \ref{prop:Xmu}. In any case, the isomorphism class of the components $Z_{1}^i$ contained in $X$ determines the possible classes of the components $Z_{0,k}^i$ containing $Y_{\beta_i,k}$.
\end{remark}

\begin{proposition} \label{prop:compassY_0} Let  $Y_{0,k}\subset X^{H_2}$ be as in Notation \ref{not:fixedH2}. Then there exist $H_1^i$-fixed components $Z_{0,k}^i$, $i=1,2,3$, containing $Y_{0,k}$ as an $H_2/H_1^i$-fixed component of weight zero, and 
$$\mathcal{C}(Y_{0,k},X,H_2)=\big\{\beta_i^{(2n+1-\dim Y_{0,k})/6}, \ \ i=0,\dots,5 \big\}.$$ 
\end{proposition}

\begin{figure}[h!]
\includegraphics[height=5cm]{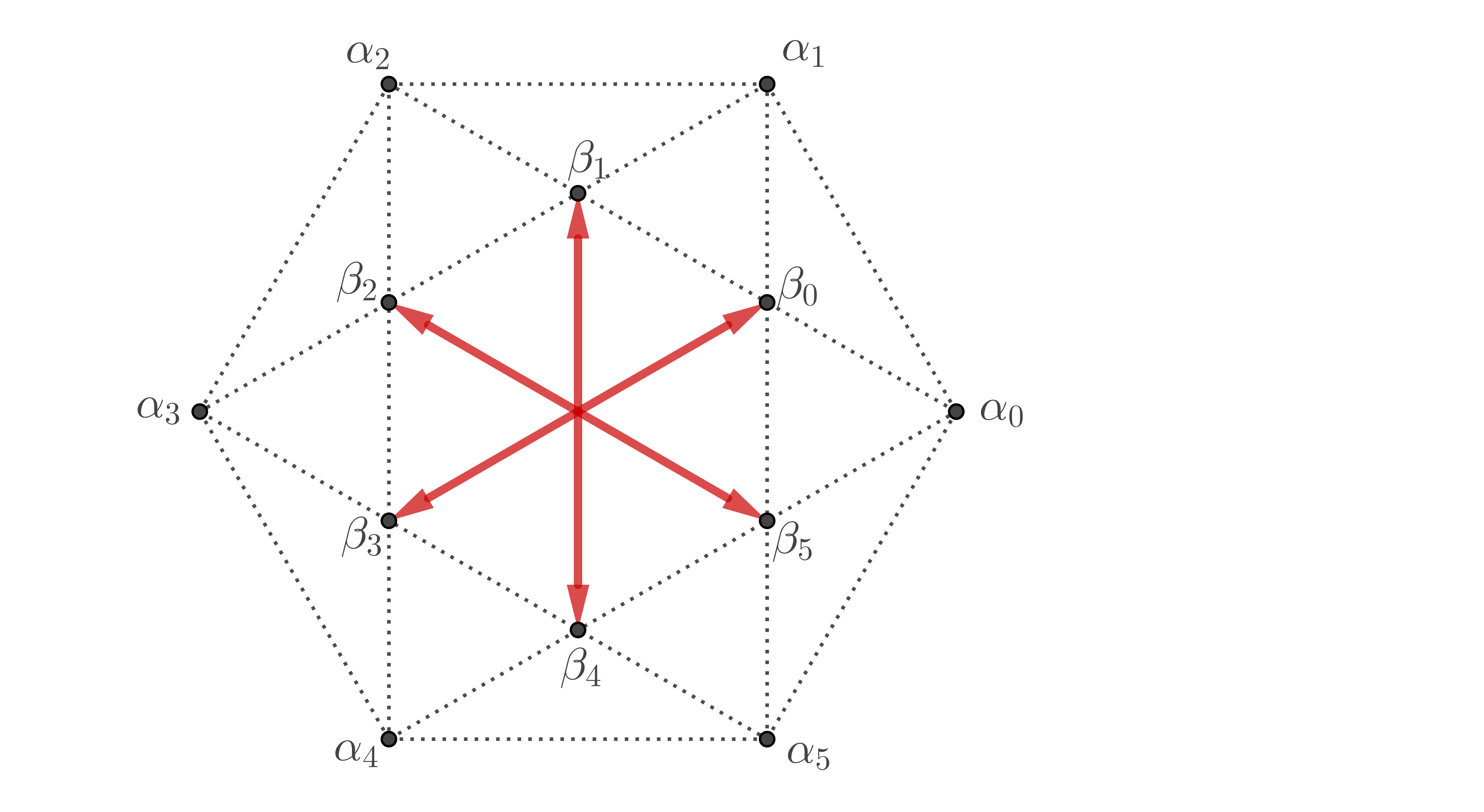}
\caption{Compass at a fixed component $Y_{0,k}$.}
\label{fig:Compass2}
\end{figure}

\begin{proof} Applying \cite[Corollary 2.14]{BWW} we know that each element in the compass $\mathcal{C}(Y_{0,k},X,H_2)$ is proportional either to a root $\alpha_i$ or $\beta_i$. We first show that no element of $\mathcal{C}(Y_{0,k},X,H_2)$ can be proportional to a  root $\alpha_i$. By contradiction, assume for instance that $\lambda\alpha_0\in \mathcal{C}(Y_{0,k},X,H_2)$, $\lambda\in\Q_{>0}$. Consider a subtorus $H_1^{\prime}\subset H_2$ corresponding to the projection $\pi'^{*}\colon M(H_2)\to M(H_1^{\prime})$, sending $\alpha_0$ to $0$, and a variety $Z^{\prime}\subset X^{H_1^{\prime}}$ which contains $Y_{0,k}$ (see Lemma \ref{lem:downgrading}). Since we have assumed that $\lambda\alpha_0\in \mathcal{C}(Y_{0,k},X,H_2)$, using Lemma \ref{lem:compass_prop} (ii) we also have $\lambda\alpha_0\in \mathcal{C}(Y_{0,k},Z',H_2/H_1')$. It follows that $Z'\supsetneq Y_{0,k}$, and that $Z'$ contains  an $H_2$-fixed point component of weight $\alpha_0$. By Lemma \ref{rem:symweights}, it contains also a fixed point component of weight $\alpha_3$, and using Lemma \ref{lem:downgrading} (iii) these two components are the isolated points $y_{0}$ and $y_{3}$. Applying Corollary \ref{cor:Y0contact},  $Z'$ is a contact manifold, on which the extremal $H_2/H_1'$-fixed components $y_0,y_3$ are Legendrian (Corollary \ref{cor:weightsBW2}), therefore $Z'\cong\P^1$ and $Y_{0,k}=\emptyset$, a contradiction.

At this point, by \cite[Corollary 2.14]{BWW}, the only possible elements in $\mathcal{C}(Y_{0,k},X,H_2)$ are proportional to the weights $\beta_i$'s and, since $Y_{0,k}\neq X$, we have at least one element, say $\beta_1'=\lambda_1\beta_1$, $\lambda_1\in\Q_{>0}$. Using Lemma  \ref{rem:symweights} we deduce that also $\beta'_4:=-\beta'_1\in \mathcal{C}(Y_{0,k},X,H_2)$, with the same multiplicity as $\beta'_1$. We take now the subtorus $H_1^1\subset H_2$ corresponding to the  the orthogonal projection $\pi_1^*\colon M(H_2)\to M(H_1)$ sending $\beta_1$ to zero (see Figure \ref{fig:hexagon2}). By Lemma \ref{lem:compass_prop} (i), the existence of $\beta'_i$ in the compass of $Y_{0,k}$ tells us that there exists a component $Z_{0,k}^1$ containing it and, as in the previous paragraph, one may show that $Z_{0,k}^1$  contains two $H_2$-fixed point components corresponding to the weights $\beta_1$ and $\beta_4$; note that since $Z_{0,k}^1$ contains $Y_{0,k}$ these components cannot be isolated points. The possibilities for these extremal fixed point components are listed in Proposition  \ref{prop:Z}; each of them determines $Z_{0,k}^1$ and, subsequently the multiplicities $\mult(\beta'_i)$ of the elements $\beta'_1$, $\beta'_4$ as elements of the compass $\mathcal{C}(Y_{0,k},Z_{0,k}^1,H_2/H_1)=\mathcal{C}(Y_{0,k},X,H_2)\cap\ker(\pi_1^*)$ (cf. Lemma \ref{lem:compass_prop} (ii)). By Proposition \ref{prop:Z}
this leaves us with the following possibilities:
\begin{table}[h!!]
\begin{tabular}{|c|c|c|c|c|c|}
\hline
 $Z_{0,k}^1$ & $Y_{0,k}$ & $\,\mult(\beta'_1)\,$ &$\dim X$ & \,$\dim Y_{0,k}$ \,& $\,\codim (Y_{0,k},X) \,$\\\hline\hline
$\, \DB_{(n-1)/2}(2)\,$ & \,$ \DB_{(n-3)/2}(2) $\,& $2$ & $\,\,2n+1\,\,$ & $\,2n-11\,$&$12$\\\hline
$\DD_{n/2}(2)$ &  $\DD_{(n-2)/2}(2)$ & $2$ &$\,2n+1\,$& $2n-11$&$12$\\\hline   
\,$\DA_5(1,5)$\, & $\DA_2(1,2) $ & $3$ & $21$& $3$ &$18$ \\\hline 
$\DD_6(2)$ & $\DA_5(1,5)$& $4$ & $33$&$9$&$24$\\\hline 
$\DE_7(1)$ & $\DE_6(2)$&$6$&$57$&$21$&$36$\\\hline
\end{tabular}
\end{table}\par
Note that, each possible $Z_{0,k}^1$ determines the isomorphism class of the fixed components $Z_{1}^i$, by Propositions \ref{prop:Z} and \ref{prop:Xmu}; in particular this determines the value of $n-1=\dim Z_{1}^i$ and consequently the dimension on $X$, that we have written in the fourth column of the table.

By repeating the same argument we can show that every element $\beta'_i=\lambda_i\beta_i$ ($\lambda_i\in\Q_{>0}$) in the compass has the same multiplicity of $\beta'_1$ and $\beta'_4$. Since the number of element in the compass, counted with multiplicity, equals $\dim X - \dim Y_{0,k}$ and this number in all the cases is six times the multiplicity of $\beta'_1$ we can conclude that all the $\beta'_i$'s appear in the compass, with multiplicity $(\dim X-\dim Y_{0,r})/6$. 

Our arguments above also show that $Y_{0,k}$ is contained in three fixed point components $Z_{0,k}^i$, $i=1,2,3$. Finally using the description of the compass at components of type $Y_{\beta_i,k}\subset Z_{0,k}^i$ (Proposition \ref{prop:compassHex}) and Lemma \ref{lem:compass_prop}, we may write 
$$\pi_i^*\mathcal{C}(Y_{0,k},X,H_2)=\cC(Z_{0,k}^i,X,H_1^i)=\pi_i^*\mathcal{C}(Y_{\beta_i,k},X,H_2)=\{(\pm 1)^{n-1-\dim Y_{\beta_i,k}}\},
$$
which suffices to show that $\lambda_i=1$ so that $\beta'_i=\beta_i$, for all $i$.
\end{proof}

\subsection{Step IV: The actions of $H_2$ on $X$ and $X_G$ have the same combinatorial data}

The results in the previous Section allow to state the following:

\begin{corollary}\label{cor:H2griddata}  
Under the assumptions of Theorem \ref{thm:LB_isolated_pts}, there exists a simple group $G'$ and a bijection $\psi:X^{H_2}\to X_{G'}^{H_2}$ such that for every $Y\in X^{H_2}$:
$$
Y\simeq \psi(Y),\quad \mu_L(Y)=\mu_{L_{G'}}(\psi(Y)),\quad \cC(Y,X,H_2)=\cC(\psi(Y),X_{G'},H_2).
$$
\end{corollary}

\begin{proof}
Given a pair $(X,L)$ as in Theorem \ref{thm:LB_isolated_pts}, we consider the downgradings $H^i_1\subset H_2\subset H$, $i=0,\dots,5$, of the action of a maximal torus $H\subset G$ presented in Step II. From Remark \ref{rem:allequal} the $H^i_1$-fixed point components $Z^i_1$ are isomorphic for all $i$, and the same holds for all the $H^2_1$-fixed point varieties $Y_{\beta_i}$.  Subsequently, all the varieties $Z_{0,k}^i$, and all the inner fixed point components $Y_{0,k}$ are determined (Proposition \ref{prop:Z} and Remark \ref{rem:allequal2}). This shows that there exists a bijection $\psi$ among the set of components of $X^{H_2}$ and the set of fixed point components of $X_{G'}^{H_2}$ for a certain simple group $G'$ (cf. Table \ref{tab:Freudenthal}), and a torus of dimension two in $G'$ that we identify with $H_2$ (chosen as in Step II).
Now we note that, by Proposition \ref{prop:compassY_0}, the above data determine the compass at every component $Y_{0,k}$, so that $\cC(Y_{0,k},X,H_2)= \cC(\psi(Y_{0,k}),X_{G'},H_2)$. In particular $2n+1=\dim X=\dim X_{G'}$ and, applying Proposition \ref{prop:compassHex} to $(X,L)$ and to $(X_{G'},L_{G'})$, the compasses $\cC(Y,X,H_2)$, $\cC(\psi(Y),X_{G'},H_2)$ are equal for every $H_2$-fixed point component $Y$ in $X$.
\end{proof}

We will show now that $X_{G'}$ is the adjoint variety of the group $G$. We will make use of the following technical lemma:

\begin{lemma}\label{lem:mistaken identity} Let $G\ne G'$ be two semisimple groups with Lie algebras $\fg,\fg'$ of type $\DA_2$, $\DB$, $\DD$, $\DE$, $\DF_4$, $\DG_2$. If their Dynkin diagrams are different, then $\fg,\fg'$ are not isomorphic as $\Mo(H_2)$-graded vector spaces. 
\end{lemma}

\begin{proof} The dimensions of the Lie algebras appearing in the statement are:
\begin{table}[h!!]
\begin{tabular}{|c|c|c|c|c|c|c|c|c|}
\hline
$\g$&{$\fa_2$}&$\fb_n$&$\fd_m$&$\fe_6$&$\fe_7$&$\fe_8$&$\ff_4$&{$\fg_2$}\\\hline\hline
$\,\dim(\g)\,$&{$8$}&$\,2n^2+n\,$&$\,2m^2-m\,$&$\,78\,$&$\,133\,$&$\,248\,$&$\,52\,$&{$\,14\,$}\\\hline
\end{tabular}
\end{table}

Note that given two positive integers $n,m$, $2n^2+n=2m^2-m$ implies that $n=(2m-1)/2$, a contradiction. Moreover, one may easily check that the only case in which $\fb_n$ or $\fd_m$ have the same dimension of another Lie algebra of the list is $\dim(\fb_6)=\dim(\fe_6)=78$. In this case, we may study explicitly the projections $\imath^*$, described in Lemma \ref{lem:reduction}, to conclude that $\dim(\fb_6)_0=24$, $\dim(\fe_6)_0=18$.
\end{proof}

\begin{corollary}\label{cor:XGXG'}
The Lie algebras $\fg$ and $\fg'$ of the groups $G$ and $G'$ are isomorphic and, in particular, $X_G\simeq X_{G'}$.
\end{corollary}

\begin{proof}
We will show that the assumptions of Proposition \ref{prop:BWW224} are fulfilled  for the actions of $H_2$ on $(X,L)$ and $(X_{G'},L_{G'})$, so that $\HH^0(X,L)$ and $\HH^0(X_{G'},L_{G'})$ are isomorphic as $H_2$-moduli. Since $\HH^0(X,L)$ is isomorphic to $\fg$ as an $H_2$-module, and the same holds for $X_{G'}$, it follows that $\fg\simeq \fg'$ as $H_2$-modules. In this way we conclude that they are also isomorphic as Lie algebras by Lemma \ref{lem:mistaken identity}.

By Corollary \ref{cor:H2griddata} we are left to check that $L_{|Y}$ and $\cN^{-\nu}(Y)$ are isomorphic to their counterparts in $X_{G'}$, for every component $Y\subset X^{H_2}$ and every $\nu\in\cC(Y,X,H)$. Since every $Y$ is contained in an $H_1^i$-fixed point component $Z$ (by Lemma \ref{lem:downgrading}), and the restriction of $L$ on these components is determined by the isomorphism class of $Z$ (see Propositions  \ref{prop:Xmu}, \ref{prop:Z} and \ref{prop:compassY_0}), the first assumption is fulfilled. 

Now we check the equality of the graded parts of the normal bundles $\cN^{-\nu}(Y)$. By Propositions  \ref{prop:compassHex}, \ref{prop:compassY_0} and Lemma \ref{lem:compass_prop}, the bundles $\cN^{-\nu}(Y)$ are $\Mo(H_2/H^i_1)$-graded parts of the normal bundle of $Y$ on the $H_1^i$-fixed point component containing it (hence they coincide with their counterparts in $X_{G'}$), for every $\nu\in \cC(Y,X,H_2)$ different from the following cases (see Figures \ref{fig:Compass1}, \ref{fig:Compass2}):
$$
(Y,\nu)=(y_{\alpha_i},-\alpha_i), \quad (y_{\alpha_i},\alpha_{i\pm 1}-\alpha_i), \quad (Y_{\beta_i,k},\beta_{i\pm 2}-\beta_i),\quad i=0,\dots,5. 
$$
In the first two cases, since $y_{\alpha_i}$ is an isolated point, there is nothing to prove. In the latter we simply note that the contact isomorphism restricts to an isomorphism $\cN^{-(\beta_{i\pm 2}-\beta_i)}(Y_{\beta_i,k})$ $\simeq$ $ L_{|Y_{\beta_i,k}}\otimes N^\vee$, where $N$ denotes the negative part (with respect to $H_2/H_1^{i+1}$) of the normal bundle $\cN_{Y_{\beta_i,k}|Z_1^{i+1}}$. This completes the proof.
\end{proof}


\subsection{Step V: Conclusion}

In this final step we will conclude the proof of Theorem \ref{thm:LB_isolated_pts} by applying Proposition \ref{prop:BWW224}, part $(C_2)$. In order to do so we need to study the fixed point components of the action on $(X,L)$ of the maximal torus $H$ of $G=\CAut^\circ(X)$, the  normal bundles of these components, the restrictions of $L$, and compare these data with those obtained for $(X_G,L_G)$ (see Section \ref{sec:torusGP}). 

We start by comparing $X^H$ and $X_G^H$, and the weights of $L$ and $L_G$.

\begin{lemma}\label{lem:grid1} Under the assumptions of Theorem \ref{thm:LB_isolated_pts}, $\dim X=\dim X_{G}$. Moreover,  the $H$-action on $X$ has only isolated fixed points,  $X^H$ is in $1$-to-$1$ correspondence with $X_G^H$, and the corresponding weights of the action on $L$ are the vertices of the root polytope $\bigtriangleup(G)$.
\end{lemma}

\begin{proof}
By Corollary \ref{cor:LSisolated} we know that $G$ is a simple group. Using Corollaries \ref{cor:H2griddata} and \ref{cor:XGXG'}, we have that $X^{H_2}$ and $X_G^{H_2}$ are in $1$-to-$1$ correspondence, and that the compasses at two corresponding components are the same. By Remark \ref{rem:isolatedpoints}, the extremal fixed points of $X$ and $X_G$ are isolated, hence the cardinalities of the compasses at these points equal  the dimensions of the varieties in question, and we get $\dim X=\dim X_G$. 

We now consider the downgrading  of the actions of $H_2$ on $X$ and $X_G$ to a general one dimensional subtorus $H_1\subset H_2$. Applying \cite[Lemma~4.1]{CARRELL} we obtain that $X^{H_1}=X^{H_2}$ and $X_G^{H_1}=X_G^{H_2}$, and we get a bijection among $X^{H_1}$  and $X_G^{H_1}$. By the properties of downgrading (Lemma \ref{lem:compass_prop} (i)), the compasses of the action of $H_1$ on $X$ and $X_G$ on corresponding components are equal, and so, by Theorem \ref{thm:BB_decomposition}, the singular homology groups of $X$ and $X_G$ are equal as well. In particular, $X$ has no odd degree homology, and the Euler characteristics of $X$ and $X_{G}$ are equal.

On the other hand, again by Theorem \ref{thm:BB_decomposition}, the Euler characteristic of $X_{G}$ is equal to the cardinality of $X_G^H$, which is bijective via the weight map $\mu_{L_G}$ to the set of long roots of $G$ (Corollary \ref{cor:adjointfix}), that is to the set of vertices of the root polytope $\bigtriangleup({G})$.
Since by Lemma \ref{lem:polyroot} one has $\bigtriangleup(G)=\bigtriangleup(X,L,H)$ and, by hypothesis, the vertices of $\bigtriangleup(X,L,H)$ are the weights of a set of isolated fixed points of $X$, it follows that the action of $H$ on $X$ cannot have
other fixed point components (because they would contribute to the Euler characteristic
of $X$). This concludes the proof.
\end{proof}
 
Since by the above lemma $X^H$ consists of isolated points, in order to apply Proposition \ref{prop:BWW224} we need only to prove that the compasses of the action of $H$ on $X$ are the same as the ones of the corresponding adjoint variety $X_G$. By abuse, we will denote corresponding fixed points in $X$ and $X_G$ by the same letter.

\begin{lemma}\label{lem:grid2}
Under the assumptions of Theorem \ref{thm:LB_isolated_pts}, for every fixed point $y\in X^{H}$ the compass $\mathcal{C}(y,X, H)$ is equal to $\mathcal{C}(y,X_G, H)$.
\end{lemma}  

\begin{proof}
We will make use of our description of the compasses of the $H_2$-action on $X$ (Propositions \ref{prop:compassHex} and \ref{prop:compassY_0}), and  of the $H_1^i$-fixed point components $Z_{0,k}^i$, $Z_1^i$, being $H_1^i$ the subtorus of $H_2$ corresponding to the orthogonal projection $\pi_i^*\colon \Mo(H_2)\to \Mo(H_1^i)$ sending $\beta_i$ to $0$ (see Figures \ref{fig:hexagon2}, \ref{fig:Compass1}, \ref{fig:Compass2}, and Notation \ref{not:Z}).

As we already observed (see the proof of Corollary \ref{cor:H2griddata}, and Corollary \ref{cor:XGXG'}), for every $i$ and every index $k$ the subvarieties $Z_{0,k}^i$, $Z_1^i$ are the same for $X$ and $X_G$. 

Given an $H$-fixed point $y\in X$, we will distinguish three cases, according to the value of $\imath^*(\mu_L(y))$, being $\imath^*\colon \Mo(H)\to \Mo(H_2)$ the projection.

Assume that $\imath^*(\mu_L(y))$ is an extremal fixed point of $\bigtriangleup(S_2)$, say $\alpha_0$. By Lemma \ref{lem:sinksourceiso} the element $-\mu_L(y)$ appears with multiplicity one in the $H$-compass at $y$. Using Lemma \ref{lem:compass_extremalpoints}, $\cC(y,X,H)$ contains two elements with multiplicity one projecting to $\alpha_1-\alpha_0$ and $\alpha_5-\alpha_0$ via $\imath^*$; by Lemma \ref{lem:compass_prop2}, these elements are necessarily $\mu_L(y_1)-\mu_L(y)$ and $\mu_L(y_5)-\mu_L(y)$, where $y_1, y_5\in X^H$ are the only fixed points satisfying that $\imath^*\mu_L(y_1)=\alpha_1$, $\imath^*\mu_L(y_5)=\alpha_5$. 
The remaining $2(n-1)$ elements are: 
\begin{enumerate}
\item $n-1$ elements in $\cC(y,X,H)\cap \ker(\pi_5^*\circ\imath^*)=\cC(y,Z_{-1}^5,H/H_1^5)$;
\item $n-1$ elements in $\cC(y,X,H)\cap \ker(\pi_0^*\circ\imath^*)=\cC(y,Z_1^0,H/H_1^0)$.
\end{enumerate}
Since the varieties $Z_{-1}^5$, $Z_1^0$ are the same for $X$ and for $X_G$, these elements of the compass are the same in both cases, and we get $\cC(y,X,H)=\cC(y,X_G,H)$. 

If $\imath^*(\mu_L(y))$ is an inner point of $\bigtriangleup(S_2)$ different from zero, say $\beta_2$, then, denoting by $Y_{\beta_2,k}$ the unique $H_2$-fixed component containing the point $y$,
 $\cC(y,X,H)$ consists of $2n+1$ elements, $2n-1$ of which can be described as follows: \begin{enumerate}
\item $n-1$ elements in $\cC(y,X,H)\cap \ker(\pi_1^*\circ\imath^*)=\cC(y,Z_{-1}^1,H/H_1^1)$;
\item $n-1-\dim Y_{\beta_2,k}$ elements in $\cC(y,Z_{-1}^0,H/H_1^0)\setminus \cC(y,Z_{-1}^1,H/H_1^1)$; note that $\cC(y,Z_{-1}^0,H/H_1^0)\cap \cC(y,Z_{-1}^1,H/H_1^1)=\cC(y,Y_{\beta_2,k},H/H_2)$; 
\item $\dim Y_{\beta_2,k}+1$ elements in $\cC(y,Z_{0,k}^2,H/H^2_1)$.
\end{enumerate}
Since the varieties $Z_{-1}^0$, $Z_{-1}^1$ and $Z_{0,k}^2$ are the same for $X$ and for $X_G$, these elements of the compass are the same in both cases. Among the elements in the compasses $\cC(y,Z_{-1}^1,H/H_1^1)$ and $\cC(y,Z_{-1}^0,H/H_1^0)$ we have two distinguished ones, appearing with multiplicity one, that project via $\imath^*$ onto $\alpha_2-\beta_2$ and $\alpha_3-\beta_2$, respectively; call them $v_2$ and $v_3$. By Lemma \ref{lem:compass_prop2}, they are necessarily the vectors $v_2=\mu_L(y_2)-\mu_L(y)$, $v_3=\mu_L(y_3)-\mu_L(y)$, where $y_2,y_3\in X^H$ are the only fixed points satisfying that $\imath^*\mu_L(y_2)=\alpha_2$, $\imath^*\mu_L(y_3)=\alpha_3$. Then we can apply Lemma \ref{rem:symweights}  to $v_2$ and $v_3$ to find  the last two elements of the compass at $y$: $-\mu_L(y_2)$, $-\mu_L(y_3)$. We conclude that $\cC(y,X,H)=\cC(y,X_G,H)$. 

Finally, in the case of a fixed point $y$ such that $\imath^*(\mu_L(y))=0$, the argument is analogous: the compass $\cC(y,X,H)$ will be the union of the compasses of the varieties $Z_{0,k}^i$ containing $y$, which are the same for $X$ and for $X_G$. 
\end{proof}



\section{High rank torus actions on contact manifolds}
\label{contact_high}

As explained in the Introduction, we use Theorem \ref{thm:LB_isolated_pts} to improve Fang's theorems (\cite{Fa1,Fa2}). In the language of projective geometry those results can be read as characterizations of (some) adjoint varieties as contact Fano manifolds whose groups of automorphisms have rank bigger than a certain bound. We reduce that bound so that it can be used to characterize most adjoint varieties (see Figure \ref{fig:fang}):

\begin{theorem}\label{thm:LBFang}
Let $(X,L)$ be a contact Fano manifold of dimension $2n+1$ with $\Pic X=\ZZ L$. Suppose
that the identity component $G$ of the group of contact automorphisms of $X$ is reductive of rank $r\geq \max(2,(n-3)/2)$. Then $(X,L)=(X_{G},L_{G})$, and $G$ is simple of one of the following types: $\DB_r$ $(r\geq 3)$, $\DD_r$ $(r\geq 4)$, $\DE_6$, $\DE_7$,  $\DF_4$, $\DG_2$.
\end{theorem}

\begin{remark} \label{rem:hypothesis_rank}
Note that the adjoint variety  of type $\DE_8$ does not satisfy  the assumptions, as in this case $\dim
X_G=57$, $n=28$, $r=8$. By Remark \ref{rem:red_dim} we may assume $n\geq 5$. On the other hand the hypothesis $r\geq 2$ is needed only for the cases $n=5,6$, in which the statement has been proved in \cite[Theorem 6.2]{RW}. Without loss of generality, we may then assume that $n\geq 7$. 
\end{remark}

The Theorem will follow from Theorem \ref{thm:LB_isolated_pts} by showing that the extremal fixed points of the action of a maximal torus $H\subset G$ are isolated. In order to do that, we will use some preliminary results. 

\begin{lemma}\label{lem:wahl}
Let $X$ be a projective manifold of dimension $d \ge 2$ with $\Pic X\simeq\ZZ$. If there exists
an action of $\CC^*$ with a fixed point component of codimension $1$ then
$X\simeq\PP^d$ and the fixed point component in question is a hyperplane.
\end{lemma}

\begin{proof} 
Let $D$ denote the codimension one fixed component. The action of $\CC^*$ determines a non-zero vector field which vanishes at the fixed points of the action, hence it gives a nonzero section in $\HH^0(X,T_X\otimes\cO_X(-D))$. Since $D$ is effective and $\Pic X\simeq\ZZ$, then $D$ is ample and the result follows by \cite[Theorem 1]{Wahl}. 
\end{proof}

Following \cite[Corollary 3.8]{BWW} the next Lemma allows us to produce sections of $L$ on $X$ by extending them from the extremal fixed point components of the action of $H$, and to prove the equality of the polytopes $\bigtriangleup(X,L,H)$ and $\Gamma(X,L,H)$: 

\begin{lemma}\label{lem:extremalh0}
Let $(X,L)$ be a contact Fano manifold of dimension $2n+1$ with $\Pic X=\ZZ L$. Suppose
that the identity component of the group of contact automorphisms $G$ is reductive with maximal torus
$H$ of rank $r\geq (n-3)/2>0$. Then, for every extremal fixed point
 component $Y$ of positive dimension we have $\dim \HH^0(Y,L_{|Y}) > 1$. Furthermore, $\Gamma(X,L,H)=\bigtriangleup(X,L,H)=\bigtriangleup(G)$. 
\end{lemma}

\begin{proof} 
As in the proof of \cite[Proposition 3.9]{BWW} we consider  a full flag of faces of
$\bigtriangleup(X,L,H)$ containing the vertex determined by the weight of $L|_Y$. Quotienting $\Mo(H)$ by the sublattices generated by the lattice points contained in these faces, we obtain a sequence of subtori:
$$H =: H^0 \supsetneq H^1 \supsetneq \dots \supsetneq H^{r-1},$$
where $H^i$ is of dimension $r-i$, and a sequence of
smooth irreducible subvarieties $Y=:Y^0\subsetneq 
Y^1\subsetneq\cdots\subsetneq Y^{r-1}$ satisfying that $Y^i$ is a fixed point component for the action of $H^i$; note that this sequence is strictly increasing by Lemma \ref{lem:downgrading}. Since each $Y^i$ is invariant by $H^j$, for $j\leq i$, then we have an action of $H^{i-1}/H^i\simeq\C^*$ on $Y^i$, with $Y^{i-1}$ as a fixed point component.

In particular $Y^{r-1}$ is associated to a facet of
$\bigtriangleup(X,L,H)$, so it is an extremal fixed point component for the action of $H^{r-1}\simeq\C^*$; and the weight of the action of $H^{r-1}$ on $L$ at the fixed component $Y^{r-1}$ is different from zero. Applying Lemma \ref{lem:sinksourceiso}, we conclude that $Y^{r-1}$ is an isotropic submanifold of $X$; therefore $\dim Y^{r-1}\leq n$. 

Assume first  that, for some $i$, we have $\dim Y^i-\dim Y^{i-1}=1$. Then, by Lemma \ref{lem:wahl},
one has $Y^i \simeq \PP^m$ for some $m$. So $Y^0$ is a positive dimensional subvariety of $\P^m$ and clearly $\dim \HH^0(Y^0,L|_{Y^0})>1$. 

Assume now that $\dim Y^i-\dim Y^{i-1}\geq 2$ for every $i$. Then
$$n-\dim Y^0 \geq \dim Y^{r-1}-\dim Y^0=\sum_{i=1}^{r-1}(\dim Y^i-\dim Y^{i-1})\geq 2(r-1),$$
 which yields, by our assumption on $r$, that $\dim Y^0\leq 5$; by \cite[Lemma 2.9 (1)]{ORSW} $Y^0$ is a Fano manifold with Picard number $1$ and \cite[Lemma 6.3]{RW} (see also \cite[Corollary~1.3]{HoS}) gives $\dim \HH^0(Y^0,L|_{Y^0})>1$.

For the second part of the statement, we note that $\HH^0(Y,L_{|Y})\ne 0$ holds trivially for every zero dimensional extremal fixed point component, hence the conclusion follows by Lemma \ref{lem:polyroot}.
\end{proof} 

With all the above materials at hand we may finally prove our main statement.  

\begin{proof} [Proof of Theorem \ref{thm:LBFang}] In view of Remark \ref{rem:hypothesis_rank}, we  assume that $n\geq 7$,  so that the rank of $G$ is bigger than or equal to $(n-3)/2$. By Lemma \ref{lem:extremalh0} one has $\Gamma(X,L,H)=\bigtriangleup(X,L,H)=\bigtriangleup(G)$. Moreover, because of Lemma \ref{lem:AutCAut} we deduce that the multiplicities of the weights at the  vertices of $\Gamma(X,L,H)$ are the same as for the adjoint representation, thus equal to $1$, that is $\HH^0(Y,L_{|Y})=1$ for every extremal fixed point component.  By Lemma \ref{lem:extremalh0} the extremal fixed components must be isolated points, and we may conclude applying Theorem \ref{thm:LB_isolated_pts}.
\end{proof}


\end{document}